\newif\ifarxiv

\arxivfalse  

\documentclass[12pt, reqno]{amsart}
\usepackage{amsfonts}
\usepackage{amsthm}
\usepackage{amscd}
\usepackage{amssymb}
\usepackage{graphics}
\usepackage{amsmath}
\usepackage[english]{babel}
\usepackage{hyperref}
\usepackage{bbm}
\usepackage{yfonts}
\usepackage{mathrsfs}
\usepackage{ upgreek }
\usepackage{color}
\usepackage{epsfig}
\usepackage{graphicx}

\allowdisplaybreaks[4]
\DeclareFontFamily{OML}{rsfs}{\skewchar\font'177}
\DeclareFontShape{OML}{rsfs}{m}{n}{ <5> <6> rsfs5 <7> <8> <9> rsfs7
<10> <10.95> <12> <14.4> <17.28> <20.74> <24.88> rsfs10 }{}
\DeclareMathAlphabet{\mathfs}{OML}{rsfs}{m}{n}

\newtheorem{prop}{Proposition}[section]

\newtheorem{thm}[prop]{Theorem}
\newtheorem{lem}[prop]{Lemma}

\newtheorem{conj}[prop]{Conjecture}

\newtheorem{defn}[prop]{Definition}
\newtheorem{rem}[prop]{Remark}
\newtheorem{prob}[prop]{Problem}
\newtheorem{exam}[prop]{Example}
\newtheorem{clm}[prop]{Claim}

\newtheorem{ques}[prop]{Question}

\newtheorem{ass}[prop]{Assumption}

\newcommand{\BE}{{\mathbb{E}}}

\newcommand{\BN}{{\mathbb{N}}}

\newcommand{\BP}{{\mathbb{P}}}

\newcommand{\BR}{{\mathbb{R}}}

\newcommand{\BZ}{{\mathbb{Z}}}

\newcommand{\FB}{{\mathfrak{B}}}

\newcommand{\CC}{{\mathcal{C}}}

\newcommand{\CE}{{\mathcal{E}}}

\newcommand{\CP}{{\mathcal{P}}}
\newcommand{\CQ}{{\mathcal{Q}}}

\newcommand{\CS}{{\mathcal{S}}}

\newcommand{\Fp}{{\mathfrak{p}}}

\newcommand{\ind}{{\mathbbm{1}}}

\newcommand{\bae}{\begin{equation}\begin{aligned}}
\newcommand{\eae}{\end{aligned}\end{equation}}

\newcommand{\om}{{\omega}}
\newcommand{\Om}{{\Omega}}
\newcommand{\si}{{\sigma}}

\newcommand{\gth}{{\theta}}

\newcommand{\ep}{{\varepsilon}}

\begin{document}

\numberwithin{equation}{section} \numberwithin{figure}{section}
\newtheorem*{theo}{Theorem}
\newtheorem*{propo}{Proposition}

\title{Random Walks on Discrete Point Processes}
\author{Noam Berger$^1$}
\thanks{$^1$The Hebrew University of Jerusalem and TU Munich}
\author{Ron Rosenthal$^2$}
\thanks{$^2$ETH Z\"urich}

\maketitle

\begin{abstract}
    We consider a model for random walks on random environments (RWRE) with a random subset of $\BZ^d$ as the vertices, and uniform transition probabilities on
    $2d$ points (two "coordinate nearest neighbors" in each of the $d$ coordinate directions). We prove that the velocity of such random walks is almost surely
    zero, give partial characterization of transience and recurrence in the different dimensions and prove Central Limit Theorem (CLT) for such random walks,
    under a condition on the distance between coordinate nearest neighbors.
\end{abstract}



\section{Introduction}


\subsection{Background}
$~$\\

Random walks on random environments is the object of intensive mathematical research for more than 3 decades. It deals with models from condensed matter physics, physical chemistry, and many other fields of research. The common subject of all models is the investigation of particles movement in inhomogeneous media. It turns out that the randomness of the media (i.e.~the environment) is responsible for some unexpected results, especially in large scale behavior. In the general  case, the random walk takes place in a countable graph $(V,E)$, but the most investigated models deals with the graph of the $d$-dimensional integer lattice, i.e., $\BZ^d$. For some of the results on those models see \cite{Ze03}, \cite{Sz00}, \cite{Hu96} and \cite{Re05}. The definition of RWRE involves two steps: First the environment is randomly chosen according to some given distribution, then the random walk, which takes place on this fixed environment, is a Markov chain with transition probabilities that depend on the environment. We note that the environment is kept fixed and does not evolve during the random walk, and that the  random walk, given the environment, is not necessarily reversible. The questions on RWRE come in two major flavors: Quenched, in which the walk is distributed according to a given typical environment, and annealed, in which the distribution of the walk is taken according to an average on the environments. There are two main differences between the quenched and the annealed laws: First the quenched is Markovian, while the annealed distribution is usually not. Second, in most models there is some additional assumption of translation invariance of the environments, which implies that the annealed law is translation invariance, while  the quenched law is not.

In contrast to most of the models for RWRE on $\BZ^d$, this work deals with non nearest neighbor random walks. In our case this is most expressed in the estimation of $\BE\left[|X_n|\right]$. Unlike nearest neighbor models we don't have an a priori estimation on the distance made in one step. Nonetheless using an ergodic theorem by Nevo and Stein we managed to bound the above and therefore to show that the estimation $\BE[|X_n|]\leq c(\om)\sqrt{n}$ still holds. The subject of non nearest neighbor random walks has not been systematically studied. For results on long range percolation see \cite{Be01}. For literature on the subject in the one dimensional case see \cite{BG08}, \cite{Br02} and \cite{CS09}. For some results on bounded non nearest neighbors see \cite{Ke84}. For some results that are valid in that general case see \cite{Va02}. For recurrence and transience criteria CLT and more for random walks on random point processes, with transition probabilities between every two points decaying in their distance, see \cite{CFG8} and the references therein. Our model also has the property that the random walk is reversible. For some of the results on this topic see \cite{BBHK}, \cite{BP07}, \cite{MP07} and \cite{sidoravicius2004quenched}.


\subsection{The Model}
$~$\\

We start by defining the random environment of the model which will be a random subset of $\BZ^d$, the $d$-dimensional lattice of integers (we also refer to such random environment as a random point process). Denote $\Om=\{0,1\}^{\BZ^d}$ and let $\mathfrak{B}$ be the Borel $\si$-algebra (with respect to the product topology) on $\Om$. For every $x\in \BZ^d$ let $\gth_x:\Om\rightarrow\Om$ be the shift along the vector $x$, i.e.~for every $y\in\BZ^d$ and every $\om\in\Om$ we have $\gth_x(\om)(y)=\om(x+y)$. In addition let $\CE=\CE(d)=\{\pm e_i\}_{i=1}^d$, where $e_i$ is a unit vector along the $i^{th}$ principal axes.

Throughout this paper we assume that $Q$ is a probability measure on $\Om$ satisfying the following:

\begin{ass}\label{Assumptions}$~$\\
    \begin{enumerate}

        \item $Q$ is stationary and ergodic with respect to each of the translations $\{\gth_{e_i}\}_{i=1}^{d}$.\\

        \item $Q(\CP(\om) = \emptyset)<1$, where $\mathcal{P}(\om)=\{x\in\BZ^d:\om(x)=1\}$.
    \end{enumerate}
\end{ass}

Let $\Om_0=\{\om\in\Om:\om(0)=1\}$. It follows from Assumptions \ref{Assumptions} that $Q(\Om_0)>0$ and therefore we can define a new probability measure $P$ on $\Om_0$ as the conditional probability of $Q$ on $\Om_0$, i.e.:
\begin{equation}
    P(B)=Q(B|\Om_0)=\frac{Q(B\cap \Om_0)}{Q(\Om_0)},\quad\forall B\in\mathfrak{B}. \label{P_definition}
\end{equation}
We denote by $\BE_Q$ and $\BE_P$ the expectation with respect to $Q$ and $P$ respectively.\\

\begin{clm}\label{The_model_claim}
    For $Q$ almost every $\om\in\Om$, every $v\in\BZ^d$ and every vector $e\in\CE$ there are infinitely many $k\in\BN$ such that $v+ke\in\CP(\om)$.
\end{clm}

\begin{proof}
    Denote $\Om_v=\{\om\in\Om:v\in\CP(\om)\}$ and notice that $\mathbbm{1}_{\Om_v}\in L^1(\Om,\mathfrak{B},Q)$. Since $\gth_e$ is measure preserving and ergodic with respect to $Q$, by Birkhoff's Ergodic Theorem
    \[
        \lim_{n\rightarrow\infty}\frac{1}{n}\sum_{k=0}^{n-1}{\gth_e^k\mathbbm{1}_{\Om_v}}=\BE_Q\left[\mathbbm{1}_{\Om_v}\right]=Q(\Om_v)=Q(\Om_0)>0,\quad Q~a.s.
    \]
    Consequently, there exist $Q$ almost surely infinitely many integers such that $\gth_e^k\mathbbm{1}_{\Om_v}=1$, and therefore infinitely many $k\in\BN$ such that $v+ke\in\CP(\om)$.
\end{proof}

The following function measures the distance of "coordinate nearest neighbors" from the origin in an environment:

\begin{defn}
    For every $e\in\CE$ we define $f_e:\Om\rightarrow\BN^+$ by
    \begin{equation}
        f_e(\om)=\min\{k>0:\gth_e^k(\om)(0)=\om(ke)=1\}.
    \end{equation}
    Note that $f_e$ and $f_{-e}$ have the same distribution with respect to $Q$.\label{f_e_defn}
\end{defn}

For every $v\in\BZ^d$ define $N_v(\om)$ to be the set of the $2d$ "coordinate nearest neighbors" in $\om$ of $v$, one for each direction (see Figure \ref{fig:fig1}). More precisely $N_v(\om)=\bigcup_{e\in\CE}\left\{v+f_e(\theta_v(\om))e\right\}$. By Claim \ref{The_model_claim} $f_e(\theta_v(\om))$ is $Q$ almost surely well defined and therefore $N_v(\om)$ is $Q$ almost surely a set of $2d$ points in $\BZ^d$.\\

\begin{figure}
    \epsfig{figure=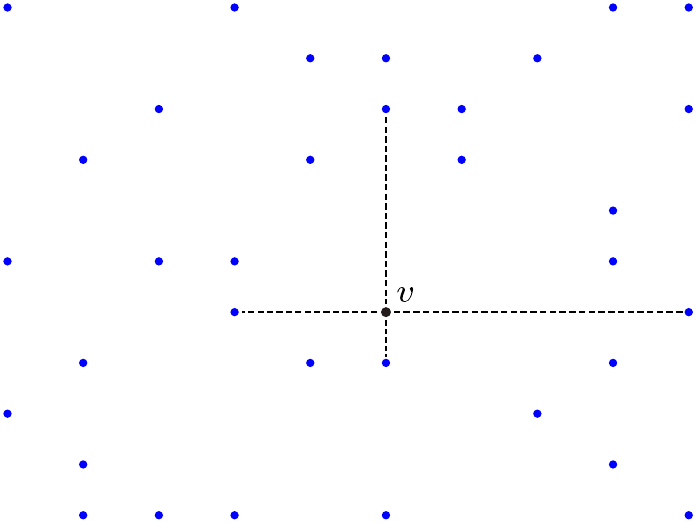, width=0.38\textwidth, clip}
    \smallskip
    \caption{An example for coordinate nearest neighbors}
    \label{fig:fig1}
\end{figure}

We now turn to define the random walk on environments. Fix some $\om\in\Om_0$ such that $|N_v(\om)|=2d$ for every $v\in \CP(\om)$. The random walk on the environment $\om$ is defined on the probability space $((\BZ^d)^\BN,\mathcal{G},P_\om)$, where $\mathcal{G}$ is the $\si$-algebra generated by cylinder functions, as the Markov chain taking values in $\CP(\om)$ with initial condition
\begin{equation}
    P_\om(X_0=0)=1, \label{inital_condition}
\end{equation}
and transition probability
\begin{equation}
    P_\om(X_{n+1}=u|X_n=v)=\left\{
    \begin{array}{cc}
        0&~~~u\notin N_v(\om) \\
        \frac{1}{2d}&~~~u\in N_v(\om) \\
    \end{array}
    \right. .\label{transition_probability}
\end{equation}

The distribution of the random walk according to this measure is called the quenched law of the random walk, and the corresponding expectation is denoted by $E_\om$.

Finally, since for each $G\in\mathcal{G}$, the map $\om\mapsto P_\om(G)$ is $\mathfrak{B}$ measurable, we may define the probability measure $\bold{P}=P\otimes P_\om$ on $(\Om_0\times(\BZ^d)^\BN,\mathfrak{B}\times\mathcal{G})$ by
\[
    \bold{P}(B\times G)=\int_{B}{P_\om(G)P(d\om)},\quad\forall B\in\mathfrak{B},~\forall G\in\mathcal{G}.
\]
The marginal of $\bf{P}$ on $(\BZ^d)^\BN$, denoted by $\BP$, is called the annealed law of the random walk and its expectation is denoted by $\BE$.

In the proof of the  high dimensional Central Limit Theorem we will assume in addition to assumptions \ref{Assumptions} the following:

\begin{ass} $~$\\

    (3) There exists $\ep_0>0$ such that $E_P\left[f_e^{2+\ep_0}\right]<\infty$ for every coordinate direction
    $e\in\CE$. \label{assumption3}\\
\end{ass}


\subsection{Examples}
$~$\\

Before turning to state and prove theorems regarding the model we give a few examples for distributions of points in $\BZ^2$ which satisfy the above conditions.

\begin{exam}[Bernoulli percolation]
The first obvious example for point process which satisfies the above conditions is the Bernoulli vertex percolation. Fix some $0<p<1$ and declare every point $v\in\BZ^d$ to be in the environment independently with probability $p$.
\end{exam}

\begin{exam}[Infinite component of supercritical percolation]
Fix some $d\geq 2$ and denote by $p_c(\BZ^d)$ the critical value for Bernoulli edge percolation on $\BZ^d$. For every $p_c(\BZ^d)<p\leq 1$ there exists with probability one a unique infinite component in $\BZ^d$, which we denote by $\CC^\infty=\CC^\infty(\om)$. We can now define the environment by $\CP(\om)=\CC^\infty(\om)$, i.e., the points in the environment are exactly the points of the unique infinite cluster of the percolation process.
\end{exam}

\begin{exam}
We denote by $\{r_n\}_{n\in\BN}$ and $\{p_n\}_{n\in\BN}$ two sequences of positive numbers, the first satisfies $\lim_{n\rightarrow\infty}r_n=\infty$ and the second satisfies $\lim_{n\rightarrow\infty}p_n=0$ and $p_n<1$ for every $n\in\BN$. We define the environment by the following procedure: For  every $v\in\BZ^d$ and $n\BN$ delete the ball of radius $r_n$ centered at $v$ with probability $p_n$. If the sequence $p_n$ converge fast enough to zero and the sequence $r_n$ converge slow enough to infinity, this procedure yields a random point process that satisfy the model assumptions.
\end{exam}

\begin{exam}[Random interlacement]
Fix some $d\geq 3$. In \cite{sznitman2010vacant} Sznitman introduced the model of random interlacement in  $\BZ^d$. Informally this is the union of traces of simple random walks in $\BZ^d$. The random interlacement in $\BZ^d$ is a distribution on points in $\BZ^d$ which satisfies the above conditions (see \cite[Theorem 2.1]{sznitman2010vacant}).
\end{exam}


\subsection{Main Results}
$~$\\

Our main goal is to study the behavior of random walks in this model. The results are summarized in the following theorems:\\

(1) \textbf{Law of Large Numbers:}  For $P$ almost every $\om\in\Om_0$, the limiting velocity of the random walk exists and equals zero. More precisely:

\begin{thm}\label{LLN}
    Let $(\Om,\FB,Q)$ be a $d$-dimensional discrete point process satisfying assumption \ref{Assumptions}, then
    \[
        \BP\left(\left\{\lim_{n\rightarrow\infty}{\frac{X_n}{n}}=0\right\}\right)=1.
    \]
\end{thm}

(2) \textbf{Recurrence Transience Classification:} We give a partial classification of recurrence-transience for random walks on discrete point processes. The precise statements are:

\begin{prop}
    Any one dimensional random walk on a discrete point process satisfying assumption \ref{Assumptions} is $\BP$ almost surely recurrent. \label{tran_recu1}
\end{prop}

\begin{thm}
    Let $(\Om,\FB,Q)$ be a two dimensional discrete point process satisfying assumption \ref{Assumptions} and assume there exists a constant $C>0$ such that
    \begin{equation}
        \sum_{k=N}^{\infty}k\cdot P(f_{e_i}=k)\leq\frac{C}{N},\quad\forall i\in\{1,2\},~\forall N\in\BN, \label{Cauchy_tail_ass}
    \end{equation}
    which in particular holds whenever $f_{e_i}$ has a second moment for $i\in\{1,2\}$. Then the random walk is $\BP$ almost surely recurrent.\label{tran_recu2}
\end{thm}

\begin{thm}
    Fix $d\geq 3$ and let $(\Om,\FB,Q)$ be a $d$-dimensional discrete point process satisfying assumption \ref{Assumptions}. Then the random walk is $\BP$ almost surely transient.\label{tran_recu3}
\end{thm}

(3) \textbf{Central Limit Theorems} - We prove that one-dimensional random walks on discrete point processes satisfy a Central Limit Theorem. We also prove that in dimension $d\geq 2$, under the additional assumption \ref{assumption3}, the random walks on a discrete point process satisfy a Central Limit Theorem.

\begin{thm}
    Let $(\Om,\FB,Q)$ be a one-dimensional discrete point process satisfying assumption \ref{Assumptions}. Then $\BE_P[f_{1}]<\infty$ and for $P$ almost every $\om\in\Om_0$ \begin{equation}
        \lim_{n\rightarrow\infty}{\frac{X_n}{\sqrt{n}}}\overset{D}{=}N(0,\BE_P^2[f_{1}]),
    \end{equation}
    where $N(0,a^2)$ denotes the normal distribution with zero expectation and variance $a^2$, and the limit is in distribution. \label{CLT1}
\end{thm}

\begin{rem}
    Note that for one-dimensional random walks on discrete point processes CLT holds even without the assumption that the variance of $f_1$ is finite. In particular the diffusion constant is given by the square of $\BE_P[f_1]$.
\end{rem}

\begin{thm}
    Fix $d\geq 2$ and let $(\Om,\FB,Q)$ be a $d$-dimensional discrete point process satisfying assumptions \ref{Assumptions} and \ref{assumption3}, then for $P$ almost every $\om\in\Om_0$
    \begin{equation}
        \lim_{n\rightarrow\infty}{\frac{X_n}{\sqrt{n}}}\overset{D}{=}N(0,D),
    \end{equation}
    where $N(0,D)$ is a $d$-dimensional normal distribution with zero expectation and covariance matrix $D$ that depends only on $d$ and the distribution of $P$. As before the limit is in distribution.\label{CLT2}
\end{thm}

\textbf{Structure of the paper.} Sect.~\ref{sec:induced_shift_properties} collects some facts about the Markov chain on environments and some ergodic results
related to it. It is based on previously known material. In Sect.~\ref{sec:LLN} we deal with the proof of Law of Large Numbers and in Sect.~\ref{sec:one_dim_CLT} with the one dimensional Central Limit Theorem. The Recurrence Transience classification is discussed in Sec.~\ref{sec:Transience_Recurrence}. The novel parts of the high dimensional Central Limit Theorem proof (asymptotic behavior of the random walk, construction  of the corrector and sublinear bounds on the corrector) appear in Sect.~\ref{sec:asymptotice_behavior}-\ifarxiv\else  \ref{sec:essential_sublinearity}\fi. The actual proof of the high dimensional Central Limit Theorem is carried out in Sect.~\ref{sec:high_dim_CLT}. Finally Sect.~\ref{sec:conj_and_ques} contains further discussion, some open questions and conjectures.


\section{The Induced shift And The Environment Seen From The Random Walk}\label{sec:induced_shift_properties}
$~$\\
The content of this section is a standard textbook material. The form in which it appears here is taken from Section 3 of \cite{BB06}. Even though it had all been known before, \cite{BB06} is the best existing source for our purpose.

Fix some $e\in\CE$. Since by Claim \ref{The_model_claim} $f_e$ is $Q$ almost surely finite we can define the induced shift $\si_e:\Om_0\rightarrow\Om_0$ by
\[
    \si_e(\om)=\gth^{f_e(\om)}_e\om.
\]

\begin{thm}
    For every $e\in\CE$, the induced shift $\si_e:\Om_0\rightarrow\Om_0$ is measure preserving and ergodic with respect to $P$. \label{induced_shift_ergodicy}
\end{thm}

\ifarxiv
\else

The proof of Theorem \ref{induced_shift_ergodicy} can be found in \cite{BB06} (Theorem 3.2).

\fi

Our next goal is to prove that the Markov chain on environments (i.e.~the Markov chain given by the environment viewed from the particle) is ergodic. Let $\Xi=\Om_0^{\BZ}$ and define $\mathscr{B}$ to be the product $\si$-algebra on $\Xi$. The space $\Xi$ is a space of two-sided sequences  $(\ldots,\om_{-1},\om_0,\om_1,\ldots)$, the trajectories of the Markov chain on environments. Let $\mu$ be the measure on $(\Xi,\mathscr{B})$ such that for any $B\in\mathscr{B}^{2n+1}$ (coordinates between $-n$ and $n$),
\[
    \mu\big((\om_{-n},\ldots,\om_n)\in B\big)=\int_{B}{P(d\om_{-n})\Lambda(\om_{-n},d\om_{-n+1})\ldots\Lambda(\om_{n-1},d\om_n)},
\]
where $\Lambda:\Om_0\times\mathfrak{B}\rightarrow [0,1]$ is the Markov kernel defined by
\begin{equation}
    \Lambda(\om,A)=\frac{1}{2d}\sum_{x\in\BZ^d}{\mathbbm{1}_{\{x\in N_0(\om)\}}\mathbbm{1}_{\{\gth_x\om\in A\}}}=\frac{1}{2d}\sum_{e\in\CE}\ind_{\{\si_e(\om)\in A\}}. \label{Lambda}
\end{equation}
Note that the sum is finite since for $Q$ almost every $\om\in\Om$ there are exactly $2d$ elements in $N_0(\om)$. Because $P$ is preserved by $\Lambda$ (see  Theorem \ref{induced_shift_ergodicy}), the finite dimensional measures are consistent, and therefore by Kolmogorov's theorem $\mu$ exists and is unique. One can see from the definition of $\mu$ that $\{\gth_{X_k}(\om)\}_{k\geq 0}$ has the same law under $\BP$ (the annealed law) as $(\om_0,\om_1,\ldots)$ has under $\mu$. Let $\widetilde{T}:\Xi\rightarrow\Xi$ be the shift defined by $(\widetilde{T}\om)_n=\om_{n+1}$. The definition of $\widetilde{T}$ implies that it is measure preserving. In fact the following also holds:

\begin{prop}
    $\widetilde{T}$ is ergodic with respect to $\mu$. \label{ergodic_particle_view}
\end{prop}

\ifarxiv
\else

As before, a proof can be found in section $3$ of \cite{BB06} (Proposition 3.5).

\fi

\begin{thm}
    Let $f\in L^1(\Om_0,\mathfrak{B},P)$. Then for $P$ almost every $\om\in\Om_0$
    \[
        \lim_{n\rightarrow\infty}{\frac{1}{n}\sum_{k=0}^{n-1}{f\circ\gth_{X_k}(\om)}=\BE_P[f]},\quad P_\om~almost~surely.
    \]
    Similarly, if $f:\Om\times\Om\rightarrow\BR$ is measurable with $\BE\left[f(\om,\gth_{X_1}\om)\right]<\infty$, then
    \[
        \lim_{n\rightarrow\infty}{\frac{1}{n}\sum_{k=0}^{n-1}{f(\gth_{X_k}\om,\gth_{X_{k+1}}\om)}=\BE\left[f(\om,\gth_{X_1}\om)\right]}
    \]
    for $P$ almost every $\om$ and $P_\om$ almost every trajectory of $(X_k)_{k\geq 0}$. \label{Thm_mutual_ergodic}
\end{thm}

\begin{proof}
    Recall that $\{\gth_{X_k}(\om)\}_{k\geq 0}$ has the same law under $\BP$ as $(\om_0,\om_1,\ldots)$ has under $\mu$. Hence, if  $g(\ldots,\om_{-1},\om_0,\om_1,\ldots)=f(\om_0)$ then
    \[
        \lim_{n\rightarrow\infty}{\frac {1}{n}\sum_{k=0}^{n-1}{f\circ\gth_{X_k}}}\overset{D}{=}\lim_{n\rightarrow\infty}{\frac{1}{n}\sum_{k=0}^{n-1}{g\circ\widetilde{T}^k}}.
    \]
    The latter limit exists by Birkhoff's Ergodic Theorem (we have already seen that $\widetilde{T}$ is ergodic) and equals $E_\mu[g]=\BE_P[f]$ almost surely. The second part follows from the first.
\end{proof}


\section{Law of Large Numbers}\label{sec:LLN}

This section is devoted to the proof of Theorem \ref{LLN}, the Law of Large Numbers for random walks on discrete point processes.

\begin{proof}[Proof of Theorem \ref{LLN}]
    Using linearity, it is enough to prove that
    \[
        \BP\left(\left\{\lim_{n\rightarrow\infty}{\frac{X_n\cdot e}{n}}=0\right\}\right)=1,\quad\forall e\in\CE.
    \]
    Fix some $e\in\CE$ and define $S(k)=\max\left\{n\geq 0:\sum\limits_{m=0}^{n-1}{f(\si_e^m(\om))}<k\right\}$. Because $f_e$ is positive, if $\BE_P[f_e]=\infty$, then
    \begin{equation}
        \lim_{n\rightarrow\infty}{\frac{1}{n}\sum_{k=0}^{n-1}{f_e(\si_e^k(\om))}}=\infty,\quad P~a.s \ifarxiv\else \nonumber\fi \label{lim_f}
    \end{equation}
%
%
    and therefore
    \[
        \lim_{k\rightarrow\infty}{\frac{S(k)}{k}}=0,\quad P~a.s.
    \]
    However, since $S(k)=\sum\limits_{j=0}^{k-1}{\ind_{\Om_0}(\gth_e^j(\om))}$, by Birkhoff's Ergodic Theorem and Assumption \ref{Assumptions}
    \[
        \lim_{k\rightarrow\infty}{\frac{S(k)}{k}}=\lim_{k\rightarrow\infty}{\frac{1}{k}\sum_{j=0}^{k-1}{\ind_{\Om_0}(\gth_e^j(\om))}}=Q(\Om_0)>0,\quad P~a.s.
    \]
    Thus $\BE_P[f_e]<\infty$. Applying Birkhoff Ergodic Theorem once more we get
    \begin{equation}
        \lim_{n\rightarrow\infty}{\frac{1}{n}\sum_{k=0}^{n-1}{f_e(\si_e^k(\om))}}=\BE_P[f_e]<\infty,\quad P~a.s.\label{Finite_sum_of_changes}
    \end{equation}

    The stationarity of $P$ with respect to $\sigma_e$ implies that $P(f_{-e}(\om)=k)=P(f_e(\si_e^{-1}(\om))=k)=P(f_e(\om)=k)$, and therefore
    \begin{equation}
        \BE_P[f_e]=\BE_P[f_{-e}]. \label{Expectation_equality}
    \end{equation}

    Let $g_e:\Om\times\Om\rightarrow\BZ$ be defined by:
    \footnotesize{\begin{equation}
    \begin{aligned}
        g_e(\om,\om')=
        \left\{
        \begin{array}{cll}
            f_e(\om) & \quad &\text{if }\om'=\si_e(\om) \\
            -f_{-e}(\om) & \quad & \text{if }\om'=\si_{-e}(\om) \\
            0 & \quad & \text{otherwise}
        \end{array}
        \right..
    \end{aligned}
    \nonumber
    \end{equation}}\normalsize{}
    Observing that $g_e$ is measurable and recalling \eqref{Expectation_equality} we get that
    \[
        \BE_P\left[E_\om(g_e(\om,\gth_{X_1}\om))\right]=\BE_P\left[\frac{1}{2d}f_e(\om)-\frac{1}{2d}f_{-e}(\om)\right]=0.
    \]
    Thus for $P$ almost every $\om\in\Om_0$ and $P_\om$ almost every random walk $\{X_k\}_{k\geq 0}$, we have by Theorem \ref{Thm_mutual_ergodic} \begin{equation}
    \begin{aligned}
        \lim_{n\rightarrow\infty}{\frac{X_n\cdot e}{n}} &=\lim_{n\rightarrow\infty}{\frac{1}{n}\sum_{k=1}^{n}{(X_k-X_{k-1})\cdot e}}\\
        &= \lim_{n\rightarrow\infty}{\frac{1}{n}\sum_{k=0}^{n-1}{g_e(\gth_{X_k}\om,\gth_{X_{k+1}}\om)}}= \BE_P\left[E_\om(g_e(\om,\gth_{X_1}\om))\right]=0. \nonumber
    \end{aligned}
    \end{equation}
\end{proof}


\section{One Dimensional Central Limit Theorem}\label{sec:one_dim_CLT}

This section is devoted to the proof of Theorem \ref{CLT1} - Central Limit Theorem of one-dimensional random walks on discrete point processes. The basic observation of the proof is the fact that random walk on discrete point processes in one dimension is in fact a simple random walk on $\BZ$ with stretched edges. Combining this with the fact that $E_P[f_1]<\infty$ implies the result. We turn to make this into a more precise argument:

\begin{proof}[Proof of Theorem \ref{CLT1}]
    Denote $e=1$. Given an environment $\om\in\Om_0$ and a random walk $\{X_k\}_{k\geq 0}$, we define the simple one-dimensional random walk $\{Y_k\}_{k\geq 0}$ associated with $\{X_k\}_{k\geq 0}$ by:
    \begin{equation}
    Y_{k}=\begin{cases}
        \sum\limits_{j=1}^{k}\frac{X_{j}-X_{j-1}}{\left|X_{j}-X_{j-1}\right|} & \quad k\geq1\\
        0 & \quad k=0
    \end{cases}.
    \nonumber
    \end{equation}
    Since $\{Y_k\}_{k\geq 0}$ is a simple one dimensional random walk on $\BZ$, it follows from the Central Limit Theorem that for $P$ almost every $\om\in\Om_0$
    \begin{equation}
        \lim_{n\rightarrow\infty}{\frac{1}{\sqrt{n}}\cdot Y_n}\overset{D}{=}N(0,1).\label{eq:simple_CLT}
    \end{equation}

    Given an environment $\om\in\Om_0$ and $n\in\BZ$ let $\Fp_n=\Fp_n(\om)$ be the $n^{th}$ point in $\CP(\om)$ (with respect to $0$). More precisely denote \begin{equation}
        \Fp_n
        =\begin{cases}
            \sum\limits_{k=0}^{n-1}{f_{e}(\si_{e}^{k}\om)}\vphantom{\underset{|}{A}} & \quad n>0 \\
            0 & \quad n=0\\
            \sum\limits_{k=-1}^{-n}{f_{e}(\si_{e}^{k}\om)} & \quad n<0
        \end{cases}.
    \end{equation}

    For every $a\in\BR\backslash\{0\}$ and $P$ almost every $\om\in\Om_0$ we have
    \[
        \lim_{n\rightarrow\infty}{\frac{1}{\sqrt{n}}\Fp_{\lfloor a\sqrt{n}\rfloor}}
        =a\cdot\lim_{n\rightarrow\infty}{\frac{1}{a\sqrt{n}}\sum_{k=0}^{\lfloor
        a\sqrt{n}\rfloor}{f_{e}(\si_e^k\om)}}=a\cdot \BE_P[f_{e}].
    \]
    In fact the last argument also holds trivially for $a=0$, i.e.~for every $a\in\BR$
    \begin{equation}
        \lim_{n\rightarrow\infty}{\frac{1}{\sqrt{n}}\Fp_{\lfloor a\sqrt{n}\rfloor}}=a\cdot \BE_P[f_{e}]. \label{limit_CLT_constant}
    \end{equation}

    Using \eqref{eq:simple_CLT} and \eqref{limit_CLT_constant} we get that for $P$ almost every $\om\in\Om_0$ and every $\ep>0$
    \footnotesize{
    \begin{equation}
        \begin{array}{rcl}
            \lim\limits_{n\rightarrow\infty}P_\om\left(\frac{\Fp_{Y_n}}{\sqrt{n}}\leq a\right)&\leq& \lim\limits_{n\rightarrow\infty}P_\om\left(\frac{\Fp_{Y_n}}{\sqrt{n}}\leq a~,~\frac{Y_n}{\sqrt{n}}>\frac{a}{E_P[f_e]}+\ep\right)+P_\om\left(\frac{Y_n}{\sqrt{n}}\leq\frac{a}{E_P[f_e]}+\ep\right)\\
            &\leq&\lim\limits_{n\rightarrow\infty}P_\om\left(\frac{1}{\sqrt{n}}\Fp_{\lfloor\left(\frac{a}{E_P[f_e]}+\ep\right)\sqrt{n}\rfloor}\leq a\right)+P_\om\left(\frac{Y_n}{\sqrt{n}}\leq\frac{a}{E_P[f_e]}+\ep\right)\\
            &=&\Phi\left(\frac{a}{E_P[f_e]}+\ep\right)
        \end{array}
    \nonumber
    \end{equation}
    }\normalsize{}
    where $\Phi$ is the standard normal cumulative distribution function. A similar argument gives that  $\lim_{n\rightarrow\infty}P_\om\left(\frac{\Fp_{Y_n}}{\sqrt{n}}\leq a\right)\geq \Phi\left(\frac{a}{E_P[f_e]}-\ep\right)$ for every $\ep>0$. Observing that $X_n=\Fp_{Y_n}$ and recalling that $\ep>0$ was arbitrary we get
    \begin{equation}
        \lim_{n\rightarrow\infty}P_\om\left({\frac{X_n}{\sqrt{n}}}\leq a\right)=\Phi\left(\frac{a}{\BE_P[f_{e}]}\right),
    \end{equation}
    as required.
\end{proof}


\section{Transience and Recurrence}\label{sec:Transience_Recurrence}

Before continuing to deal with the Central Limit Theorem in higher dimensions, we turn to a discussion on transience-recurrence of random walks on discrete point processes.


\subsection{One-dimensional case}

Here we wish to prove the recursive behavior of the one-dimensional random walk on discrete point processes (Proposition \ref{tran_recu1}). This follows from the same coupling introduced in order to prove the CLT.

\begin{proof}[Proof of Proposition \ref{tran_recu1}]
    Using the notation from the previous section, since $Y_n$ is a one-dimensional simple random walk, it is recurrent $\BP$ almost surely. Therefore we have  $\#\{n:Y_n=0\}=\infty~~\BP$ almost surely, but since $X_n=\Fp_{Y_n}$ and $\Fp_0=0$ this implies $\#\{n:X_n=0\}=\infty$, $\BP$ almost surely. Thus the random walk is recurrent.
\end{proof}


\subsection{Two-dimensional case}

In this section we deal with the two-dimensional case. The proof is based on the correspondence of random walks to electrical networks. Recall that an electrical network is given by a triple $G=(V,E,c)$, where $(V,E)$ is an unoriented graph and $c:E\rightarrow (0,\infty)$ is a conductance field. We start by recalling the Nash-Williams criterion for recurrence of random walks:

\begin{thm}[Nash-Williams criterion]
    A set of edges $\Pi$ is called a cutset for an infinite network $G=(V,E,c)$ if there exists some vertex $a\in V$ such that every infinite simple path from $a$ to infinity must include an edge in $\Pi$. If $\{\Pi_n\}$ is a sequence of pairwise disjoint finite cutsets in a locally finite infinite graph $G$, each of which separates $a\in V$ from infinity and $\sum_n\left(\sum_{e\in\Pi_n}c(e)\right)^{-1}=\infty$, then the random walk induced by the conductances $c$ is recurrent.
\end{thm}

For a proof of the Nash-Williams criterion and some background on the subject see \cite{DS84} and \cite{LP97}. The following definition will be used in the proof:

\begin{defn}
    Let $(\widetilde{\Om},\widetilde{\mathfrak{B}},\widetilde{P})$ be a probability space. We say that a random variable $X:\widetilde{\Om}\rightarrow [0,\infty)$ has a Cauchy tail if there exists a positive constant $C$ such that $\widetilde{P}\left(X\geq n\right)\leq \frac{C}{n}$ for every $n\in \BN$.
\end{defn}

Note that if $\widetilde{E}[X]<\infty$, then X has a Cauchy tail.

In order to prove Theorem \ref{tran_recu2} we will need the following lemmas taken from \cite{Be01}.

\begin{lem}[\cite{Be01} Lemma 4.1]
    Let $\{f_i\}_{i=1}^{\infty}$ be identically distributed (not necessarily independent) positive random variables, on a probability space  $(\widetilde{\Om},\widetilde{\mathfrak{B}},\widetilde{P})$, that have a Cauchy tail. Then, for every $\epsilon>0$, there exist $K>0$ and $N\in\BN$ such that  for every $n>N$
    \[
        \widetilde{P}\left(\frac{1}{n}\sum_{i=1}^{n}{f_i}>K\log n\right)<\epsilon.
    \]
    \label{Cauchy_tail_lemma}
\end{lem}

\begin{lem}[\cite{Be01} Lemma 4.2]
    Let $A_n$ be a sequence of events such that $\tilde{P}(A_n)>1-\epsilon$ for all sufficiently large $n$, and let $\{a_n\}_{n=1}^\infty$ be a sequence such that $\sum_{n=1}^{\infty}{a_n}=\infty.$ Then $\sum_{n=1}^{\infty}{a_n\mathbbm{1}_{A_n} }=\infty$ with probability of at least $1-\ep$. \label{lem2_2dim_recurrent}
\end{lem}

We also need the following definition:

\begin{defn}
    Assume $G=(V,E)$ is a graph such that $V\subset\BZ^2$ and $E$ is a set of edges, each of them is parallel to some axis, but may connect non nearest neighbors in $\BZ^2$. For an edge $e\in E$ we denote by  $e^+,e^-\in V$ the end points of $e\in E$. In order for this to be well defined we assume that if $(e^+-e^-)\cdot e_i\neq 0$ then $(e^+-e^-)\cdot e_i>0$. Note that by the assumption on the edges in $e$ the value of $e^+-e^-$ is non zero in exactly one coordinate.
\end{defn}

\begin{proof}[Proof of Theorem \ref{tran_recu2}]
    The idea of the proof is to construct for every $\om\in\Om$ an electrical network which satisfy the Nash-Williams criterion and induce the same law on the random walk as the law of the random walk on $\om$, $P$-a.s. Since $P$ is a marginal of $Q$ it is enough to construct a network which satisfy the criterion for $Q$ almost every $\om\in\Om$. For every $\om\in\Om$, we define the corresponding network with conductances $G(\om)=(V(\om),E(\om),c(\om))$ via the following three steps (See figure \ref{fig:network_construction} for an illusration): Step 1. Define $G_1(\om)=(V_1(\om),E_1(\om),c_1(\om))$ to be the network induced from $\om$ with all conductances equal to $1$. More precisely we define
    \begin{equation}
        \begin{array}{lll}
            V_1(\om)&=&\CP(\om)\\
            E_1(\om)&=&\left\{\{x,y\}\in V_1\times
            V_1:y\in\{x\pm f_{e_1}(\om)e_1,x\pm f_{e_2}(\om)e_2\}\right\}\\
            c_1(\om)(e)&=&1,\quad\forall e\in E_1
        \end{array}.
        \nonumber
    \end{equation}
    Note that the continuous time random walk induced by the network $G_1(\om)$ (cf. \cite{DS84,LP97}) is indeed the random walk introduced in \eqref{transition_probability} when $0\in\CP(\om)$.

    Step 2. Define $G_2(\om)$ to be the network generated from $G_1(\om)$ by "cutting" every edge of length $k$ into $k$ edges of length $1$, giving conductance $k$ to each part. A small technical problem with "cutting" the edges is that vertical and horizontal edges may cross each other in a point that doesn't  belong to $\CP(\om)$. In order to avoid this we give the following formal definition which is a bit cumbersome:
    \begin{equation}
        \begin{array}{lll}
            V_2(\om)&=&V_2^1(\om)\biguplus
            V_2^2(\om)\subset\BZ^2\times\{0,1\}\\
            E_2(\om)&=&E_2^1(\om)\biguplus E_2^2(\om)\\
        \end{array},
        \nonumber
    \end{equation}
    where
    \begin{equation}
        V_2^i(\om)=\left\{(x,i):
        \begin{array}{cc}
            \exists~e\in E_1(\om)~,~\exists~0\leq k\leq |e^+-e^-|_1\\
            \text{such that } (e^+-e^-)\cdot e_i\neq 0~,~x=e^-+ke_i
        \end{array}
        \right\}
        \nonumber
    \end{equation}
    and
    \begin{equation}
        E^i(\om)=\left\{\left\{(v,i),(w,i)\right\}:
        \begin{array}{cc}
            \exists~e\in E_1(\om)~,~\exists~0\leq k< |e^+-e^-|_1 \text{ such that }\\
            ~(e^+-e^-)\cdot e_i\neq 0~,~v=e^-+ke_i~,~w=e^-+(k+1)e_i
        \end{array}
        \right\}.
        \nonumber
    \end{equation}
    We also define the conductance $c'(\om)(e)$ of an edge $e\in E'(\om)$ to be $k$, given that the length (i.e. $|e^+-e^-|_1$) of the original edge it was part of was $k$. Step 3. Define $G(\om)$ to be the graph obtained from $G_2(\om)$ by identifying two vertices if they are of the form $(v,1)$ and $(v,2)$ for some $v\in\CP(\om)$.
    \begin{figure}
        \epsfig{figure=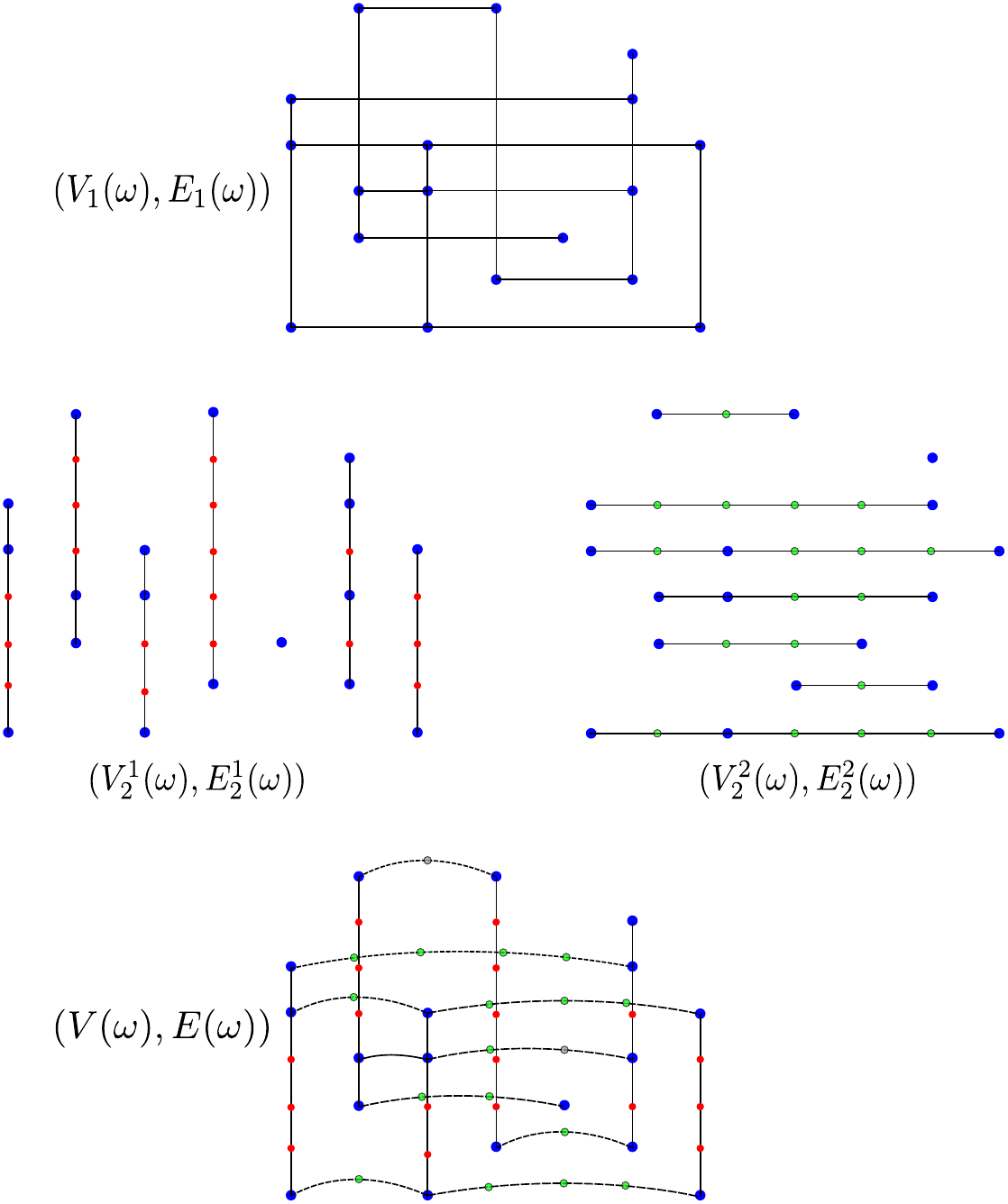, width=0.6\textwidth, clip}
        \smallskip
        \caption{Construction of the network in two dimensions}
        \label{fig:network_construction}
    \end{figure}
    Note that by a standard analysis of conductances, see e.g.~\cite{DS84}, it is clear that the random walk on the new network is transient if and only if the original random walk is transient. Thus we turn to prove the recurrence of the random walk on the new graph. This is done using the Nash-Williams Criterion. Let $\Pi_n$ be the set of edges exiting the box $[-n,n]^2\times \{1,2\}$ in the graph $G(\om)$. The sets $\Pi_n$ define a sequence of pairwise disjoint cutsets in the network $G(\om)$, i.e., a set of edges that any infinite simple path starting at the origin must cross. Next we wish to estimate the conductances in the network $G$. Fix some $e\in E$ such that $(e^+-e^-)\cdot e_i\neq 0$ and note that the distribution of $c(e)$ is the same for all edges in direction $e_i$. For $\om\in\Om$ we denote by $\text{len}_{e_i}(\om)$ the length of the interval containing the origin in direction $e_i$, where in the case that the origin belongs to the point process we define $\text{len}_{e_i}(\om)$ to be the length of the interval starting at the origin in direction $e_i$. More precisely we define $\text{len}_{e_i}(\om)=f_{e_i}(\om)+g_{e_i}(\om)$, where $g_{e_i}(\om)=\min\{n\leq 0~:~\omega(ne_i)=1\}$. In addition for $n\in\BN$ we define $l_n(\om)=l_n^{e_i}(\om)$ to be the length of the first $n^{th}$ intervals starting at the origin in direction $e_i$, i.e., $l_n(\omega)=g_{e_i}(\om)+\sum_{j=0}^{n-1}f_{e_i}\left(\sigma_{e_i}^j-1(\omega)\right)$. Using the definition of $\text{len}_{e_i}$ we have the following estimate
    \begin{equation}
        Q(c(e)=k)=Q\left(
        \begin{array}{cc}
            \text{the original edge that}\\
            \text{contained } e \text{ in } G_1(\om)\\
            \text{is of length } k
        \end{array}
        \right)=Q(\text{len}_{e_i}(\om)=k).
        \nonumber
    \end{equation}
    By Birkhoff's Ergodic Theorem the last term $Q$ almost surely equals
    \[
        \lim_{n\to\infty}\frac{1}{n}\sum_{j=0}^{n-1}\ind_{\{\text{len}_{e_i}(\theta^j \om)=k\}}.
    \]
    Since $l_n$ tends to infinity $Q$ almost surely and $g_{e_i}(\om)$ is finite $Q$ almost surely this implies
    \[
        Q(c(e)=k)=\lim_{n\to\infty}\frac{1}{l_n(\om)}\sum_{j=0}^{l_n(\om)}\ind_{\{\text{len}_{e_i}(\theta^j\om)=k\}}=\lim_{n\to\infty}\frac{n}{l_n(\om)}\cdot \frac{1}{n}\sum_{\footnotesize{j=-g_{e_i}(\om)}}^{\footnotesize{l_n(\om)-g_{e_i}(\om)}}\ind_{\{\text{len}_{e_i}(\theta^j\om)=k\}},
    \]
    which after rearrangement can be written as
    \[
        \lim_{n\to\infty}\frac{n}{l_n}\cdot \frac{1}{n}\sum_{j=0}^{n-1}k\cdot\ind_{\{f_{e_i}(\sigma_{e_i}^j(\om))=k\}}.
    \]
    Recalling that by Birkhoff's Ergodic Theorem (applied to the induced shift) we also have $P$ almost surely
    \[
        \lim_{n\to\infty}\frac{l_n}{n}=\BE_P[f_{e_i}]\quad ,\quad \lim_{n\to\infty}\frac{1}{n}\sum_{j=0}^{n-1}\ind_{\{f_{e_i}(\sigma_{e_i}^j(\om))=k\}}=P(f_{e_i}=k),
    \]
    we get
    \begin{equation}
        Q(c(e)=k)=\frac{k\cdot P(f_{e_i}=k)}{\BE_P[f_{e_i}]}.
        \label{eq:c(e)_dist}
    \end{equation}

    From \eqref{eq:c(e)_dist} and the assumption of Theorem \ref{tran_recu2}, i.e.~\eqref{Cauchy_tail_ass}, it follows that $c(e)$ has a Cauchy tail. Note that $\Pi_n$ contains $4n+2$ edges at each level, i.e., in each of the sets $\{E^i(\om)\}_{i=1,2}$,  all of them with the same distribution (by \eqref{Cauchy_tail_ass} with a Cauchy tail), though they may be dependent. By Lemma \ref{Cauchy_tail_lemma}, for every $\ep>0$ there exist $K>0$ and $N\in\BN$ such that for every $n>N$, we have
    \begin{equation}
        Q\left(\sum_{e\in\Pi_n}{c(e)}\leq K(8n+4)\log (8n+4)\right)>1-\ep. \label{Cauchy_event}
    \end{equation}

    Define $A_n$ to be the event in equation (\ref{Cauchy_event}), and $a_n=(K(8n+4)\log (8n+4))^{-1}$. Notice that  $C_{\Pi_n}\overset{_\text{def}}{=}\sum_{e\in\Pi_n}c(e)$ satisfies $\sum_{n=1}^{\infty}{{C_{\Pi_n}}^{-1}}\geq \sum_{n=N}^{\infty}{\mathbbm{1}_{A_n}\cdot a_n}.$ In addition the definition of $\{a_n\}$ implies that $\sum_{n=N}^{\infty}{a_n}=\infty$. Combining the last two facts together with \eqref{Cauchy_event} and Lemma \ref{lem2_2dim_recurrent} gives $Q\left(\sum_{n=1}^{\infty}{{C_{\Pi_n}}^{-1}}=\infty\right)\geq 1-\ep$. Since $\ep$ is arbitrary, we get that $\sum_{n=1}^{\infty}{{C_{\Pi_n}}^{-1}}=\infty$, $Q$ a.s.~and therefore in particular $P$ a.s. Thus by the Nash-Williams criterion, the random walk is $\BP$ almost surely recurrent.
\end{proof}


\subsection{Higher dimensions ($d\geq 3$)}

$~$\\

Here we prove the transience of random walks on discrete point processes in dimension $3$ or higher. The idea of the proof is to bound the heat kernel so that the Green function of the random walk will be finite. This is done by first proving an appropriate discrete isoperimetric inequality for finite subsets of  $\BZ^d$, and then using well known connections between isoperimetric inequalities to heat kernel bounds (see \cite{MP08}) to bound the heat kernel. In order to state the isoperimetric inequality we need the following definition:

\begin{defn}
    Let $x=(x_1,x_2,\ldots,x_d)$ be a point in $\BZ^d$. For $1\leq j\leq d$ denote by $\Pi^j:\BZ^d\rightarrow\BZ^{d-1}$ the projection on all but the $j^{th}$ coordinate, namely
    \[
        \Pi^j(x)=\Pi^j((x_1,x_2,\ldots,x_d))=(x_1,x_2,\ldots,x_{j-1},x_{j+1},\ldots,x_d).
    \]
\end{defn}

\begin{lem}
    There exists $C=C(d)>0$ such that for every finite subset $A$ of $\BZ^d$
    \begin{equation}
        \max_{1\leq j\leq d}\{|\Pi_j(A)|\}\geq C\cdot |A|^{\frac{d-1}{d}},
    \end{equation}
    where $|\cdot|$ denotes the cardinality of the set. \label{lem_volume_surface}
\end{lem}

Before turning to the proof we fix some notations.

\begin{defn}$~$\\
    \begin{itemize}
        \item Denote by $\CQ^d$ the quadrant of points in $\BZ^d$ all of whose entries are positive.
        \item For a point $x\in\CQ^d$ define its energy by $\mathcal{E}(x)=\sum_{j=1}^d x_j$.
        \item For a finite set $A\subset \CQ^d$ denote $\mathcal{E}(A)=\sum_{x\in A}{\mathcal{E}(x)}$.
        \item Given a finite set $A\subset \CQ^d$, $1\leq j\leq d$ and some point $y=(y_1,y_2,\ldots,y_{d-1})\in\CQ^{d-1}$ we define the $y$-fiber of $A$ in direction $j$
            \[
                A_{j,y}\overset{def}{=}\{x_j:(y_1,y_2,\ldots,y_{j-1},x_j,y_{j},\ldots,y_{d-1})\in A\}.
            \]
    \end{itemize}
\end{defn}

\begin{proof}[Proof of Lemma \ref{lem_volume_surface}]
    Assume $|A|=n$. Using translations, we can assume without loss of generality that $A\subset \CQ^d$. Next, for $1\leq j\leq d$ we define $\CS_j:2^{\CQ^d}\rightarrow 2^{\CQ^d}$ the "squeezing operator in direction $j$". The definition of $\CS_j$  is a bit complicated, however the idea is to mimic the operation of pushing the points inside each of the fibers of $A$ in direction $j$ as close to the hyperplane $x_j=0$ as possible without any of them leaving the quadrant $\CQ^d$. An illustration of $\CS_j$ operation is illustrated in Figure \ref{fig:fig3}. More formally $\CS_j$ is defined by
    \[
        \CS_j(A)=\bigcup_{y=(y_1,\ldots,y_{d-1})\in\CQ^{d-1}}\Big\{(y_1,y_2,\ldots,y_{j-1},m,y_{j},\ldots,y_{d-1})\Big\}_{m=1}^{|A_{j,y}|}.
    \]

    \begin{figure}
        \epsfig{figure=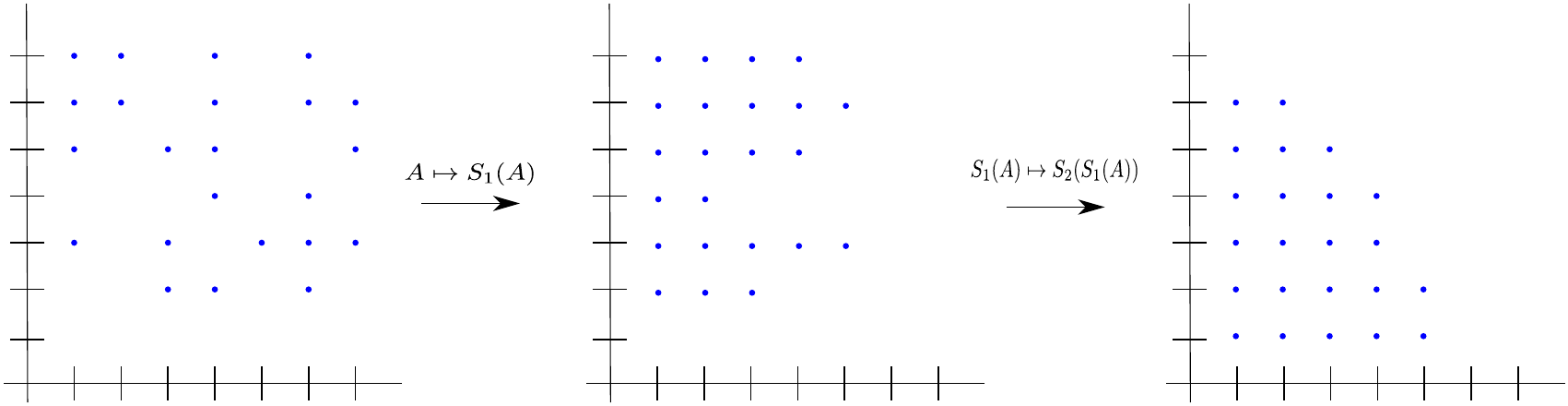, width=1\textwidth, clip}
        \smallskip
        \caption{The operator $S_j(A)$.}
        \label{fig:fig3}
    \end{figure}

    The operator $\CS_j$ satisfies the following properties:
    \begin{enumerate}

        \item The size of each fiber of $\CS_j(A)$ in direction $j$ is the same as the corresponding one for $A$.

        \item The size of $\CS_j(A)$ is the same as the size of $A$.

        \item $|\Pi^i(\CS_j(A))|\leq |\Pi^i(A)|$ for every $1\leq i,j \leq d$.

        \item $\mathcal{E}(\CS_j(A))\leq\mathcal{E}(A)$, and equality holds if and only if $\CS_j(A)=A$.

    \end{enumerate}

    Indeed,
    \begin{enumerate}
        \item This follows directly from the definition of $S^j$. Given $y=(y_1,\ldots,y_{d-1})\in\CQ^{d-1}$
        \begin{equation}
            \left|\CS_j(A)_{j,y}\right|=\left|\Big\{(y_1,y_2,\ldots,y_{j-1},m,y_{j},\ldots,y_{d-1})\Big\}_{m=1}^{|A_{j,y}|}\right|=\left|A_{j,y}\right|.
            \nonumber
        \end{equation}

        \item Since the fibers in direction $j$ of a set form a partition we get
        \[
            \left|\CS_j(A)\right|=\sum_{y\in\CQ^{d-1}}\left|\CQ_j(A)_{j,y}\right|=\sum_{y\in\CQ^{d-1}}\left|A_{j,y}\right|=|A|.
        \]

        \item For $i=j$ note that $y=(y_1,\ldots,y_{d-1})\in \Pi^j(A)$ if and only if there exists some $m\in\BN$ such that $(y_1,\ldots,y_{j-1},m,y_j,\ldots,y_{d-1})\in A$. This however is equivalent to the fact that $(y_1,\ldots,y_{j-1},1,y_j,\ldots,y_{d-1})\in \CS_j(A)$ which again is true if and only if $y=(y_1,\ldots,y_{d-1})\in \Pi^j(\CS_j(A))$. Thus $\Pi^j(A)=\Pi^j(\CS_j(A))$. Turning to the case $i\neq j$, the proof follows from the fact that we can reduce the problem into two dimensions. Without loss of generality assume that $i=1$ and $j=2$, then
            \footnotesize{
            \begin{equation}
            \begin{aligned}
                \left|\Pi^i(S_j(A))\right|=\left|\Pi^1(S_2(A))\right|&=\sum\limits_{(y_2,\ldots,y_d)\in\CQ^{d-1}}\ind_{\{\exists m\geq 1 \text{ s.t. } (m,y_2,\ldots,y_{d})\in S_2(A)\}}\\
                &=\sum\limits_{(y_3,\ldots,y_d)\in\CQ^{d-2}}\sum\limits_{y_2\in\CQ}\ind_{\{\exists m\geq 1 \text{ s.t. } (m,y_2,y_3,\ldots,y_{d})\in S_2(A)\}}\\
                &=\sum\limits_{(y_3,\ldots,y_d)\in\CQ^{d-2}}\max\limits_{y_2\in\CQ}\{|S_2(A)_{2,(y_2,\ldots,y_d)}|\}\\
                &=\sum\limits_{(y_3,\ldots,y_d)\in\CQ^{d-2}}\max\limits_{y_2\in\CQ}\{|A_{2,(y_2,\ldots,y_d)}|\}\\
                &\leq\sum\limits_{(y_3,\ldots,y_d)\in\CQ^{d-2}}\sum\limits_{y_2\in\CQ}\ind_{\{\exists m\geq 1 \text{ s.t. } (m,y_2,y_3,\ldots,y_{d})\in A\}}\\
                &=\left|\Pi^1(A)\right|=\left|\Pi^i(A)\right|.
            \end{aligned}
            \nonumber
            \end{equation}
        }\normalsize{}
        where the third equality follows from the definition of $S_2$ (see figure \ref{fig:fig3}).

        \item As before this follows from the fact that we can reduce the problem into two dimensions. By the definition of energy (and some abuse of notation)
        \footnotesize{
        \begin{equation}
        \begin{aligned}
            \CE(\CS_j(A))&=\sum\limits_{\overset{y\in \CQ^{d-1}}{y=(y_1,\ldots,y_{d-1})}}\sum\limits_{x\in \CS_j(A)_{j,y}}\CE((y_1,\ldots,y_{j-1},x,y_j,\ldots,y_{d-1}))\\
            &=\sum\limits_{y\in \CQ^{d-1}}\sum\limits_{x\in \CS_j(A)_{j,y}}(\CE(y)+x)\\
                &=\sum\limits_{y\in \CQ^{d-1}}(|\CS_j(A)_{j,y}|\CE(y)+\CE(\CS_j(A)_{j,y}))\\
            &\leq\sum\limits_{y\in \CQ^{d-1}}(|A_{j,y}|\CE(y)+\CE(A_{j,y}))\\
            &=\CE(A),
            \nonumber
        \end{aligned}
        \end{equation}
        }\normalsize{}
        where the inequality follows from the fact that any fiber of $\CS_j(A)$ in direction $j$ has the minimal energy when compared to any other fiber in the quadrant $\CQ^d$ in direction $j$ with the same number of point as $\CS_j(A)$. In particular this holds when comparing fibers of $\CS_j(A)$ and $A$ in direction $j$. Note that equality holds if and only if all the fibers of $A$ in direction $j$ are exactly the ones of $\CS_j(A)$ which implies $A=\CS_j(A)$.
    \end{enumerate}

    Let $\{a_m\}$ be the periodic sequence $1,2,\ldots,d,1,2,\ldots,d,1,2,\ldots,d,\ldots$ and define the sequence of sets $\{A_m\}$ by the recursion formula $A_0=A$ and  $A_{m+1}=\CS_{a_m}(A_m)$ for $m\geq 0$. Property (4) of the operators $\CS_j$ implies that $\CE(A_m)$ is a decreasing sequence of positive integers. Consequently, up to finite number of elements the sequence $\CE(A_m)$ is constant. Recalling once more property (4) of $\CS_j$ we get that up to finite number of sets $A_m$ is constant. Denote the constant set of the sequence by $\widetilde{A}$. The definition of the sequence $A_m$ and property (3) of $S_j(A)$ implies that
    \begin{enumerate}
        \item $\CS_j(\widetilde{A})=\widetilde{A}$ for every $1\leq j\leq d$.

        \item $|\Pi^j(\widetilde{A})|\leq |\Pi^j(A)|$ for every $1\leq j\leq d$.

        \item $|\widetilde{A}|=|A|$.
    \end{enumerate}
    The first property implies that the size of the boundary of $\widetilde{A}$ is exactly $2\sum_{i=1}^{d}|\Pi^i(\widetilde{A})|$ (see figure \ref{fig:fig3}). Using the fact that the boundary of every set of size $n$ in $\BZ^d$ is at least $C_0\cdot n^{\frac{d-1}{d}}$ for some positive constant $C_0=C_0(d)$ (see \cite{DP96}), we get that there exists a positive constant $C=C(d)$ and at least one $i_0\in\{1,2,\ldots,d\}$ such that $|\Pi^{i_0}\big(\widetilde{A}\big)|\geq C\cdot |\widetilde{A}|^{\frac{d-1}{d}}=C\cdot|A|^{\frac{d-1}{d}}$. Thus by recalling property (2) of $\widetilde{A}$, the statement holds.
\end{proof}

\vspace{4 mm}

We now turn to define the isoperimetric profile of a graph. Let $\{p(x,y)\}_{x,y\in V}$ be symmetric transition probabilities for an irreducible Markov chain on a countable state space V. We think about this Markov chain as a random walk on a weighted graph $G=(V,E,C)$, with $\{x,y\}\in E$ if and only if $p(x,y)>0$. For every $\{x,y\}\in E$ define the conductance of $(x,y)$ by  $C(x,y)=p(x,y)$. For $S\subset V$, the "boundary size" of $S$ is measured by  $|\partial S|=\sum_{s\in S}{\sum_{s'\in S^c}{p(s,s')}}$. We define $\Phi_S$, the conductance of S, by $\Phi_S:=\frac{|\partial S|}{|S|}$. Finally, define  the isoperimetric profile of the graph G,  with vertices V and conductances induced from the transition probabilities by:
\begin{equation}
    \Phi(u)=\inf\{\Phi_S:S\subset V,~|S|\leq u\}.
\end{equation}

\begin{thm}[\cite{MP08} Theorem 2]
    Let $G=(V,E)$ be a graph with countably many vertices and bounded degree. Assume there exists $0<\gamma\leq\frac{1}{2}$ such that $p(x,x)\geq\gamma$ for every $x\in V$. If
    \begin{equation}
    n\geq 1+\frac{(1-\gamma)^2}{\gamma^2}\int_{4}^{4/\ep}{\frac{4du}{u\Phi^2(u)}},\label{integral_condition}
    \end{equation}
    then
    \begin{equation}
    |p^n(x,y)|\leq\ep,
    \end{equation}
    where $p^n(x,y)$ is the probability for the Markov chain starting at $x$ to hit $y$ after $n$ steps. \label{MP_theorem}
\end{thm}

Combining Lemma \ref{lem_volume_surface} and Theorem \ref{MP_theorem} we get the following bound on the heat kernel:

\begin{prop}
    Let $p_\om^n(x,y)$ be the probability that the random walk in the environment $\om$ moves from $x$ to $y$ in $n$ steps. Then there exists a positive constant $K$ depending only on $d$, such that for every $n\in\BN$ and every $x,y\in\CP(\om)$
    \begin{equation}
        p_\om^{n}(x,y)\leq\frac{K}{n^{d/2}},\quad P~a.s.
    \end{equation}
    \label{bound_of_transitions}
\end{prop}

\begin{proof}
    We separate the discussion to the case of even times (i.e.~when $n$ is even) and odd ones starting with the first. Restricting the Markov chain only to those times, since $p^2_\om(x,x)=\frac{1}{2d}$, we can apply Theorem \ref{MP_theorem} with $\gamma=\frac{1}{2d}$. In order to get a good estimate on the heat kernel, i.e.~$p_\om^n(x,y)$, we need to show an appropriate lower bound on $\Phi(u)$. By Lemma \ref{lem_volume_surface} there exists a positive constant $C=C(d)$ with the following property: For $P$ almost every $\om\in\Om_0$ and every  $A\subset\CP(\om)$ of size $n$ at least one of the projections $\{\Pi^i(A)\}_{i=1}^d$ satisfies $\Pi^{i}(A)\geq C\cdot n^{\frac{d-1}{d}}$. Assume without loss of generality that this holds for $i=1$. Denote by $\widetilde{A}$ the "upper" boundary of $A$ in the first direction, i.e.
    \begin{equation}
        \widetilde{A}=\left\{(x_1,x_2,\ldots,x_d)~:~
        \begin{array}{c}
            (x_2,\ldots,x_d)\in\Pi^1(A)\\
            x_1=\max\{a:(a,x_2,x_3,\ldots,x_d)\in A\}
        \end{array}
        \right\}.\nonumber
    \end{equation}
    Thus $|\widetilde{A}|=|\Pi^1(A)|\geq C n^{(d-1)/d}$. By definition $|\partial A|$ equals $\frac{1}{2d}$ times the number of edges $e\in E$ with one end point in $A$ and the other in $A^c$. Since every element in $\widetilde{A}$ contributes at least one edge to the boundary we can conclude that $|\partial A|\geq \frac{1}{2d}|\widetilde{A}|$. Consequently there exists a positive constant $c_0=c_0(d)$ such that
    \begin{equation}
        \Phi(u)\geq  \frac{c_0}{u^{1/d}}. \label{Phi_estimation}
    \end{equation}
    Fix some positive constant $\tilde{K}=\tilde{K}(d)>1$ satisfying $\frac{6d(2d-1)^2\cdot4^\frac{2}{d}}{c_0^2\cdot \tilde{K}^\frac{2}{d}}<1$. From the definition of $\tilde{K}$ and using \eqref{Phi_estimation}, we get for $\ep=\frac{\tilde{K}}{n^\frac{d}{2}}$
    \begin{equation}
    \begin{aligned}
        1+(2d-1)^2\int_{4}^{4/\ep}{\frac{4du}{u\Phi^2(u)}} & \leq 1+(2d-1)^2\int_{4}^{4/\ep}{\frac{4u^{\frac{2}{d}-1}du}{c_0^2}}  \nonumber\\
        &= 1+\frac{2d(2d-1)^2\cdot 4^{\frac{2}{d}}}{c_0^2}  \ep^{-\frac{2}{d}} \nonumber\\
        &= 1+\frac{2d(2d-1)^2\cdot 4^{\frac{2}{d}}}{c_0^2\cdot \tilde{K}^\frac{2}{d}}n<1+\frac{1}{3}n. \label{isoperimetric}
    \end{aligned}
    \end{equation}
    The last term is smaller than $n$ whenever $n>1$. Thus\footnote{The fact that $\tilde{K}>1$ ensures that this also holds for $n=1$.} Theorem \ref{MP_theorem} gives that for $P$ almost every $\om\in\Om_0$ for every $x,y\in\CP(\om)$ and every $n\geq 1$
    \begin{equation}
        p_\om^{2n}(x,y)\leq \frac{\tilde{K}}{n^\frac{d}{2}}\leq \frac{2^\frac{d}{2}\tilde{K}}{(2n)^\frac{d}{2}},
        \nonumber
    \end{equation}
    which gives the result for even times with $K=2^\frac{d}{2}\tilde{K}$.

    Turning to odd times we get that for $P$ almost every $\om\in\Om_0$ every $n\in\BN$ and every $x,y\in\CP(\om)$
    \begin{equation}
        p^{2n+1}_\om(x,y)=\sum_{z\in\CP(\om)}{p_\om(x,z)p^{2n}_\om(z,y)}
        \leq \sum_{z\in\CP(\om)}p_\om(x,z)\frac{K}{(2n)^\frac{d}{2}}=\frac{K}{(2n)^\frac{d}{2}}\leq \frac{\left(\frac{3}{2}\right)^d K}{(2n+1)^\frac{d}{2}},
        \nonumber
    \end{equation}
    which completes the proof.
\end{proof}

Theorem \ref{tran_recu3} now follows immediately.

\begin{proof}[Proof of Theorem \ref{tran_recu3}]
    Since our graph is connected, it is enough to show that $\sum_{n=0}^{\infty}{p^n_\om(0,0)}$ is finite $P$ almost surely. This follows from Proposition \ref{bound_of_transitions} and the fact that $d\geq 3$.
\end{proof}


\section{Asymptotic behavior of the random walk}\label{sec:asymptotice_behavior}

This section is devoted to understanding the asymptotic behavior of $\BE[\|X_n\|]$. This estimation is used in section \ref{sec:high_dim_CLT} to prove the high dimensional Central Limit Theorem, and therefore throughout this section we also assume assumption \ref{assumption3}. The proof closely follows \cite{Ba03} with one major change: In the current model, the distance made by the random walk at each step is not bounded by $1$ as in the percolation model. Nevertheless, using an ergodic theorem of Nevo and Stein, see \cite{NS94}, we show that under assumption \ref{assumption3}, the same estimation for $\BE[\|X_n\|]$ as in percolation holds.

\begin{thm}
    Assume assumptions \ref{Assumptions} and \ref{assumption3} hold. Then there exists a random variable $c:\Om_0\rightarrow [0,\infty]$ which is finite almost surely such that for $P$ almost every $\om\in\Om_0$
    \begin{equation}
        \BE_\om[\|X_n\|]\leq c(\om)\sqrt{n},\quad\forall n\in\BN.
    \end{equation}
    \label{asymptotic_X_n}
\end{thm}

We start with some definitions:

\begin{defn}
    Fix $\om\in\Om_0$. For $n\in\BN$ we denote $p^n(x,y)=P_\om(X_n=y|X_0=x)$ and introduce the following functions, with the understanding that $0\cdot\log(0)=0$:
    \begin{itemize}

        \item The averaged two step probability $g_n:\CP(\om)\rightarrow\BR$, is given by
        \begin{equation}
            g_n(x)=\frac{1}{2}\left(p^n(0,x)+p^{n-1}(0,x)\right).
        \end{equation}

        \item Averaged two step distance $M:\BN\rightarrow\BR^+$ is defined by $M(0)=0$ and
        \begin{equation}
            M(n)=\frac{1}{2}\BE_\om\left[\|X_n\|+\|X_{n-1}\|\right]=\sum_{y\in\CP(\om)}{\|y\|g_n(y)},\quad\forall n>0.
        \end{equation}

        \item Averaged entropy $Q:\BN\rightarrow\BR^+$ is given by $Q(0)=0$ and
        \begin{equation}
            Q(n)=-\sum_{y\in\CP(\om)}{g_n(y)\log(g_n(y))},\quad\forall n>0.
        \end{equation}
    \end{itemize}
\end{defn}

The following proposition gives some inequalities which are satisfied by the functions $g_n,M$ and $Q$. Those will play a crucial rule in the proof of Theorem \ref{asymptotic_X_n}.

\begin{prop}\label{prop:asymptotic}
    There exist positive constants $c_1,c_2$ depending only on $d$ and random variables $\upsilon_3,\upsilon_4:\Om_0\rightarrow\BR$ which are $P$ almost surely finite and positive such that for every $n\in\BN$
    \begin{equation}
        Q(n)\geq \frac{d}{2}\log{(n-1)}-c_1, \label{distance_prop_1}
    \end{equation}

    \begin{equation}
        M(n)\geq c_2\cdot e^{\frac{Q(n)}{d}}, \label{distance_prop_2}
    \end{equation}

    \begin{equation}
        \sum_{x\in\CP(\om)}{\sum_{y\in\CP(\om)}{\ind_{\{y\in N_x(\om)\}}{(g_n(x)+g_n(y))\|x-y\|^2}}}<\upsilon_3, \label{distance_prop_3}
    \end{equation}
    and
    \begin{equation}
        (M(n+1)-M(n))^2\leq \upsilon_4(Q(n+1)-Q(n)). \label{distance_prop_4}
    \end{equation}
\end{prop}

\begin{rem}
    Note that we don't have any estimation on the tail of $c_3(\om)$ nor $c_4(\om)$.
\end{rem}

\begin{proof}
    For \eqref{distance_prop_1} first note that from the definition of $Q(n)$
    \begin{equation}
        Q(n)\geq \inf_{y\in\CP(\om)}(-\log(g_n(y)))=-\sup_{y\in\CP(\om)}{(\log(g_n(y)))}.\nonumber
    \end{equation}
    Proposition \ref{bound_of_transitions} implies that $g_n(y)\leq \frac{K}{(n-1)^{\frac{d}{2}}}$ for every $y\in\CP(\om)$  and therefore
    \begin{equation}
    Q(n)\geq -\log\left(\frac{K}{(n-1)^\frac{d}{2}}\right)=\frac{d}{2}\log(n-1)-\log(K),
    \end{equation}
    which gives \eqref{distance_prop_1} with $c_1=\log(K)$.

    Next we prove \eqref{distance_prop_2}. For $n\geq0$ let $D_n=B_{2^n}(0)\backslash B_{2^{n-1}}(0)$, where $B_n(0)=\{x\in\BZ^d ~:~|x|\leq n\}$. In particular $D_0=\{0\}$. Given that $0\leq a\leq 2$ we can write
    \begin{equation}
        \sum_{y\in\CP(\om)}{e^{-a\|y\|}} \leq \frac{1}{2}\sum_{n=0}^{\infty}{\sum_{y\in D_n}{e^{-a\cdot 2^n}}}\leq \sum_{n=0}^{\infty}{e^{-a\cdot 2^n}\cdot c_{2.1}\cdot 2^{nd}}\leq c_{2.2}\cdot a^{-d},
        \label{distance_inequality_1}
    \end{equation}
    where $c_{2.2}=c_{2.2}(d)>0$ depends only on $d$. Indeed, the first inequality is obvious, the second inequality follows from the fact that the set of points in $\CP(\om)$ with distance greater than $2^{n-1}$ and less than $2^n$ is bounded by the number of points in $\BZ^d$ with those properties, which is less  than a constant times $2^{nd}$.
   \ifarxiv
\else
        The proof of the last inequality follows by separating the series into two parts, up to some $n_0=\left\lceil\frac{1}{e^{\frac{a}{2d}-1}}\right\rceil$ and starting from $n_0$,  and then bounding the second one by a geometric series. More formal proof of this inequality can be found in the detailed version of this paper on the Arxiv,  see \cite{RR10}.
    \fi
    Since for every $u>0$ and $\lambda\in\BR$ the inequality $u(\log(u)+\lambda)\geq -e^{-1-\lambda}$ holds, by taking $\lambda=a\|y\|+b$ with $a\leq 2$ and $u=g_n(y)$ we get
    \begin{equation}\label{eq:one_more}
    \begin{aligned}
        -Q(n)+aM(n)+b&=\sum_{y\in\CP(\om)}{g_n(y)\left(\log(g_n(y))+a\|y\|+b\right)}\\
        &\geq-\sum_{y\in\CP(\om)}{e^{-1-a\|y\|-b}=-e^{-1-b}\sum_{y\in\CP(\om)}{e^{-a\|y\|}}}.
    \end{aligned}
    \end{equation}

    Note that we actually used the last inequality only for those $y\in\CP(\om)$ such that $g_n(y)>0$, and for $y\in\CP(\om)$ such that $g_n(y)=0$ we used the fact that $0\geq -e^{-1-a\|y\|-b}$. Combining \eqref{eq:one_more} and \eqref{distance_inequality_1} gives
    \begin{equation}
        -Q(n)+aM(n)+b\geq -e^{-1-b} c_{2.2} a^{-d}.
        \label{ditance_inequality_2.1}
    \end{equation}
    Since for sufficiently large $n$ we have
    \begin{equation}
        M(n)=0\cdot g_n(0)+\sum_{y\in\CP(\om),y\neq 0}{d(0,y)g_n(y)}\geq \sum_{y\in\CP(\om),y\neq 0}{g_n(y)}=1-g_n(0)\geq \frac{1}{2},
        \nonumber\label{distance_inequality_2.2}
    \end{equation}
    we can choose $a=\frac{1}{M(n)}$ and $b=d\cdot \log M(n)$, which together with \eqref{ditance_inequality_2.1}  gives
    \begin{equation}
        -Q(n)+1+d\cdot\log M(n)\geq -e^{-1} c_{2.2}=-c_{2.3}. \nonumber
    \end{equation}
    Note that as before $c_{2.3}$ is a positive constant that depends only on $d$. Rearranging the last inequality we get that there exists a constant $c_2=c_2(d)>0$ such that $M(n)\geq c_2\cdot e^\frac{Q(n)}{d}$.

    Turning to the prove \eqref{distance_prop_3} we first note that the sum in \eqref{distance_prop_3} can be rewritten as
    \small{
    \begin{equation}
    \begin{aligned}
        \sum_{x,y\in\CP(\om)}{\ind_{\{y\in
        N_x(\om)\}}(g_n(x)+g_n(y))\|x-y\|^2}
        &=2\sum_{x\in\CP(\om)}{g_n(x)\sum_{y\in N_x(\om)}{\|x-y\|^2}}\\
        &= 2\sum_{e\in\CE}{\sum_{x\in\CP(\om)}{g_n(x)f_e^2(\gth^{x}\om)}}\\
        &=2\sum_{e\in\CE}{\left(E_\om[f_e^2\circ
        \gth^{X_n}]+E_\om[f_e^2\circ \gth^{X_{n-1}}]\right)}.
    \end{aligned}
    \label{eq:original_sum}
    \end{equation}
    }\normalsize{}
    In order to show the sum is finite, we use a Theorem by Nevo and Stein proved in \cite{NS94}, however before we can state it some additional definitions are needed:

    Given a countable group $\Gamma$ define $\ell^1(\Gamma)=\left\{\mu\in\Gamma^\BR ~:~ \sum_{\gamma\in\Gamma}|\mu(\gamma)|<\infty\right\}$. Let $(X,\mathfrak{B},m)$ be a standard Lebesgue probability space, and assume $\Gamma$ acts on $X$ by measurable automorphisms preserving the probability measure $m$. This action induces a representation of $\Gamma$ by isometries on the $L^p(X)$ spaces, $1\leq p\leq \infty$, and this representation can be extended to  $\ell^1(\Gamma)$ by $(\mu f)(x)=\sum_{\gamma\in\Gamma}{\mu(\gamma)f(\gamma^{-1}x)}$. Let $\mathfrak{B}_1=\{A\in\mathfrak{B}:m(\gamma A\bigtriangleup  A)=0~~\forall \gamma\in\Gamma\}$ denote  the sub $\si$-algebra of invariant sets, and denote by $E_1$ the conditional expectation with respect to $\mathfrak{B}_1$. We call a sequence $\nu_n\in \ell^1(\Gamma)$ a pointwise ergodic sequence in $L^p$ if, for any action of $\Gamma$ on a Lebesgue space X which  preserves a probability measure and for every $f\in L^p(X)$, $\nu_nf(x)\rightarrow E_1[f(x)]$ for $m$ almost every $x\in X$, and in the norm of $L^p(X)$. If  $\Gamma$ is finitely generated, let $S$ be a finite generating symmetric set, i.e.~$S=S^{-1}$ which doesn't include the identity element $e$. $S$ induces a length function on $\Gamma$, given by $|\gamma|=|\gamma|_S=\min\{n:\gamma=s_1s_2\ldots s_n~,~s_i\in S\}$, and $|e|=0$. We can therefore define the following sequences:

    \begin{defn} \label{strange_notations}
    $~$\\ \vspace{-5mm}
    \begin{enumerate}
        \item[(i.)] $\tau_n=(\# S_n)^{-1}\sum_{w\in S_n} w$, where $S_n=\{w:|w|=n\}$.

        \item[(ii.)] $\tau '_n=\frac{1}{2}(\tau_n+\tau_{n+1})$.

        \item[(iii.)] $\mu_n=\frac{1}{n+1}\sum_{k=0}^{n}{\tau_k}$.

        \item[(iv.)] $\beta_n=(\# B_n)^{-1}\sum_{w\in B_n}w$, where $B_n=\{w:|w|\leq n\}$.

    \end{enumerate}
    \end{defn}

    We can now state the theorem:

    \begin{thm} [Nevo, Stein 94]
        Consider the free group $F_r$, $r\geq 2$ and let $S$ be a set of free generators and their inverses. Then:
        \begin{enumerate}
            \item[1.] The sequence $\mu_n$ is a pointwise ergodic sequence in $L^p$, for all $1\leq p< \infty$.

            \item[2.] The sequence $\tau '_n$ is a pointwise ergodic sequence in $L^p$, for $1<p<\infty$.

            \item[3.] $\tau_{2n}$ converges to an operator of conditional expectation with respect to an $F_r$-invariant sub $\si$-algebra. $\beta_{2n}$ converges to the operator $E_1+\frac{r-1}{r}E$, where E is a projection disjoint from $E_1$. Given $f\in L^p(X)$, $1<p<\infty$, the convergence is pointwise almost everywhere, and in the $L^p$ norm.
        \end{enumerate}
        \label{Birkhoff-generalization}
    \end{thm}

    Let $F$ be the (free) group generated by the induced shifts, let $\{Y_n\}$ be a simple random walk on it and $S_k=\{v\in F~:~|v|=k\}$. Then,
    \begin{equation}
        \begin{array}{rcl}
        E_\omega [f_e^2\circ \theta^{X_n}] & = &\sum_{v\in F}P(Y_n=v)f_e^2\circ \theta^{v}\\
        & = &\sum_{k=0}^\infty P(Y_n\in S_k)\frac{1}{|\#S_k|}\sum_{v\in S_k} f_e^2\circ \theta^v\\
        & = & \sum_{k=0}^\infty P(Y_n\in S_k) \tau_k\circ f_e^2,
        \end{array}
        \nonumber
    \end{equation}
    and therefore
    \small{
    \begin{equation}
        \begin{array}{rcll}
            E_\omega [f_e^2\circ \theta^{X_n}]+E_\omega [f_e^2\circ \theta^{X_{n-1}}] & = & \sum_{k=0}^\infty P(Y_n\in S_k)\tau_k\circ f_e^2+P(Y_{n-1}\in S_k)\tau_k\circ f_e^2 &\\
            &\leq & \sum_{k=1}^\infty \Big(P(Y_n\in S_k)+P(Y_{n-1}\in S_{k-1})\Big) \Big(\tau_k\circ f_e^2+\tau_{k-1}\circ f_e^2\Big)&\\
            && + P(Y_n\in S_0)f_e^2&\\
            & = & \sum_{k=1}^{\infty} 2\Big(P(Y_n\in S_k)+P(Y_{n-1}\in S_{k-1})\Big)\tau_{k-1}'\circ f_e^2+ P(Y_n\in S_0)f_e^2&
        \end{array}
        \nonumber
    \end{equation}
    }
    By assumption \ref{assumption3} there exists some $1<p<\infty$ such that $f_e^2\in L^p(\Om_0)$ for every coordinate direction $e\in\CE$. Using Theorem \ref{Birkhoff-generalization} and the ergodicity of $P$ it follows that $\sup_k\{|\tau_k'\circ f_e^2|\}$ is bounded by some constant $\upsilon_{3.1}(\omega)$ which is finite $P$ almost surely, and therefore the sum in \eqref{eq:original_sum} is bounded by
    \footnotesize{
    \begin{equation}
        \begin{array}{rcll}
            2\sum_{e\in\CE}{\left(E_\om[f_e^2\circ \gth^{X_n}]+E_\om[f_e^2\circ \gth^{X_{n-1}}]\right)} &\leq & 4\upsilon_3(\omega) \sum_{k=1}^\infty \Big(P(Y_n\in S_k)+P(Y_{n-1}\in S_{k-1})\Big)&\\
            && +  2\upsilon_3(\omega) P(Y_n\in S_0)&\\
            &=&8\upsilon_3(\omega).&
        \end{array}
        \nonumber
    \end{equation}
    }
    Consequently, the original sequence is bounded by $\upsilon_3(\om)=8\upsilon_{3.1}(\om)$ $P$ almost surely.

    Finally we turn to prove \eqref{distance_prop_4}. By the definition of $M(n)$
    \begin{equation}
        M(n+1)-M(n)=\sum_{y\in\CP(\om)}{(g_{n+1}(y)-g_n(y))\|y\|}. \nonumber
    \end{equation}
    Using the discrete Gauss Green formula, this sum can be written as
    \begin{equation}\label{Green_formula}
        -\frac{1}{4d}\sum_{x,y\in\CP(\om)}{\ind_{\{y\in N_x(\om)\}}(\|y\|-\|x\|)(g_n(y)-g_n(x))}.
    \end{equation}
    Indeed, three different sum rearrangements (recalling all sums are finite and that $|N_x(\om)|=2d<\infty$ for every point $x\in\CP(\om)$) give
    \begin{equation}
        \begin{array}{rcll}
            \sum_{y\in\CP(\om)}{(g_{n+1}(y)-g_n(y))\|y\|} &=& -\frac{1}{4d}&\left[ ~~2d\sum_{y\in\CP(\om)}{\|y\|g_n(y)}+2d\sum_{x\in\CP(\om)}{\|x\|g_n(x)}\right.\\
            &&&~\left.-2d\sum_{y\in\CP(\om)}{\|y\|g_{n+1}(y)}-2d\sum_{x\in\CP(\om)}{\|x\|g_{n+1}(x)}\right]\\
            &=&-\frac{1}{4d}&\left[~~\sum_{y\in\CP(\om)}{\|y\|g_n(y)\sum_{x\in\CP(\om)}{\ind_{y\in N_x(\om)}}}\right.\\
            &&&~+\sum_{x\in\CP(\om)}{\|x\|g_n(x)\sum_{y\in\CP(\om)}{\ind_{y\in N_x(\om)}}}\\
            &&&~-\sum_{y\in\CP(\om)}{\|y\|\sum_{x\in\CP(\om)}{\ind_{y\in N_x(\om)}g_n(x)}}\\
            &&&~\left.-\sum_{x\in\CP(\om)}{\|x\|\sum_{y\in\CP(\om)}{\ind_{y\in N_x(\om)}g_n(y)}}\right]\\
            &=&-\frac{1}{4d}&\sum_{x,y\in\CP(\om)}\left[\ind_{y\in N_x(\om)}\|y\|g_n(y)-\ind_{y\in N_x(\om)}\|x\|g_n(y)\right.\\
            &&&~~\quad\quad\quad\quad\left.-\ind_{y\in N_x(\om)}\|y\|g_n(x)+\ind_{y\in N_x(\om)}\|x\|g_n(x)\right]\\
            &=&-\frac{1}{4d}&\sum_{x,y\in\CP(\om)}{\ind_{\{y\in N_x(\om)\}}(\|y\|-\|x\|)(g_n(y)-g_n(x))}.
        \end{array}
        \nonumber
    \end{equation}

    Using the last presentation for $M(n+1)-M(n)$ and the triangle inequality gives
    \begin{equation}
        |M(n+1)-M(n)|\leq \frac{1}{4d}\sum_{x,y\in\CP(\om)}{\ind_{\{y\in N_x(\om)\}}\|x-y\|\left|g_n(y)-g_n(x)\right|}.\nonumber
    \end{equation}
    Applying Cauchy Schwartz inequality to the r.h.s we get
    \begin{equation}
        \begin{array}{rcll}
            |M(n+1)-M(n)|&\leq&\frac{1}{4d}&\left(\sum_{x,y\in\CP(\om)}{\ind_{\{y\in N_x(\om)\}}(g_n(x)+g_n(y))\|x-y\|^2}\right)^\frac{1}{2}\nonumber\\
            &&~~\cdot &\left(\sum_{x,y\in\CP(\om)}{\ind_{\{y\in N_x(\om)\}}\frac{(g_n(y)-g_n(x))^2}{g_n(y)+g_n(x)}}\right)^\frac{1}{2}.\nonumber
        \label{distance_inequality_3.1}
        \end{array}
    \end{equation}

    The first sum in the r.h.s is the same as \eqref{distance_prop_3} and therefore is bounded by some random variable $\upsilon_3=\upsilon_3(\om)$ which is positive and finite $P$ almost surely. Thus
    \begin{equation}
        |M(n+1)-M(n)|\leq \upsilon_3(\om)\left(\sum_{x,y\in\CP(\om)}{\ind_{\{y\in N_x(\om)\}}\frac{(g_n(y)-g_n(x))^2}{g_n(y)+g_n(x)}}\right)^\frac{1}{2}.
        \nonumber
    \end{equation}
    The fact that $\frac{(u-v)^2}{u+v}\leq(u-v)\left(\log(u)-\log(v)\right)$ for every $u,v>0$ yields
    \footnotesize{
    \begin{equation}
        |M(n+1)-M(n)|\leq \upsilon_3(\om)\left(\sum_{x,y\in\CP(\om)}{\ind_{\{y\in N_x(\om)\}}\Big(g_n(y)-g_n(x)\Big)\Big(\log(g_n(y))-\log(g_n(x))\Big)}\right)^\frac{1}{2}
        \nonumber
    \end{equation}
    }\normalsize{}
    which by applying the discrete Gauss Green formula the other way around equals
    \footnotesize{
    \begin{equation}
        \sqrt{4d}\upsilon_3(\om)\left(-\sum_{y\in\CP(\om)}{\Big(\log(g_n(y))+1\Big)\Big(g_{n+1}(y)-g_n(y)\Big)}\right)^\frac{1}{2}.
        \nonumber
    \end{equation}
    }\normalsize{}
    Finally, since $1-x+\log(x)\leq 0$ for all $x>0$, the last term is bounded by
    \footnotesize{
    \begin{equation}
        \sqrt{4d}\upsilon_3(\om)\left(-\sum_{y\in\CP(\om)}{\Big(g_{n+1}(y)-g_n(y)\Big)\log(g_n(y))+g_{n+1}(y)
        \log\left(\frac{g_{n+1}(y)}{g_n(y)}\right)}\right)^\frac{1}{2}=\upsilon_4\Big(Q(n+1)-Q(n)\Big)^\frac{1}{2},\nonumber
    \end{equation}
    }\normalsize{}
    where $\upsilon_4=(\sqrt{4d}\upsilon_3)^2$.
\end{proof}
\vspace{0.2cm}

\begin{proof}[Proof of Theorem \ref{asymptotic_X_n}]
    Define $R:\BN\rightarrow\BR$ by
    \begin{equation}
        R(n)=\frac{1}{d}\Big(Q(n)-\frac{d}{2}\log(n-1)+c_1\Big),
    \end{equation}
    for $n>1$ and $R(1)=0$. By \eqref{distance_prop_2} for sufficiently large $n$
    \begin{equation}
        M(n)\geq c_2\cdot e^\frac{Q(n)}{d}=c_2\cdot e^{R(n)+\frac{c_1}{d}+\frac{1}{2}\log(n-1)}=c_{5.1}e^{R(n)}\sqrt{n-1}
    \label{distance_inequality_last1}
    \end{equation}
    with $c_{5.1}$ some positive constant depending only on $d$. On the other hand by Proposition \ref{prop:asymptotic}
    \begin{equation}
        \begin{array}{rcl}
            M(n) & = &\sum_{k=1}^{n}\left(M(k)-M(k-1)\right) \nonumber \\
            & \leq & \sqrt{c_4}\sum_{k=1}^{n}{\Big(Q(k)-Q(k-1)\Big)^\frac{1}{2}} \nonumber \\
            & \leq & c_{5.2}\sum_{k=3}^{n}{\Big(Q(k)-Q(k-1)\Big)^\frac{1}{2}} \nonumber \\
            & = &c_{5.2}\sqrt{d}\sum_{k=3}^{n}{\left(R(k)-R(k-1)+\frac{1}{2}\log\left(\frac{k-1}{k-2}\right)\right)^{\frac{1}{2}}}.
            \nonumber
        \end{array}
    \end{equation}
    Denote $c_{5.3}=c_{5.2}\sqrt{d}$. Since $(a+b)^\frac{1}{2}\leq b^\frac{1}{2}+\frac{a}{(2b)^\frac{1}{2}}$ the r.h.s can be bounded by
    \begin{equation}
        \begin{array}{rcl}
            &&c_{5.3}\sum_{k=3}^{n}{\left[\frac{1}{\sqrt{2}}\log^\frac{1}{2}\left(\frac{k-1}{k-2}\right)+
            \frac{R(k)-R(k-1)}{\log^\frac{1}{2}\left(\frac{k-1}{k-2}\right)}\right]}\\
            &=& c_{5.3}\sum_{k=3}^{n}{\frac{1}{\sqrt{2}}\log^\frac{1}{2}\left(\frac{k-1}{k-2}\right)}+
            c_{5.3}\sum_{k=3}^{n}{\left[\frac{R(k)}{\log^\frac{1}{2}\left(\frac{k}{k-1}\right)}-\frac{R(k-1)}{\log^\frac{1}{2}\left(\frac{k-1}{k-2}\right)}
            \right]}\\
            &&-c_{5.3}\sum_{k=3}^{n}{R(k)\left[\frac{1}{\log^\frac{1}{2}\left(\frac{k}{k-1}\right)}-\frac{1}{\log^\frac{1}{2}\left(\frac{k-1}{k-2}\right)}
            \right]}\\
            &\leq & c_{5.3}\sum_{k=3}^{n}{\frac{1}{\sqrt{2}}\log^\frac{1}{2}\left(\frac{k-1}{k-2}\right)}+
            c_{5.3}\sum_{k=3}^{n}{\left[\frac{R(k)}{\log^\frac{1}{2}\left(\frac{k}{k-1}\right)}-\frac{R(k-1)}{\log^\frac{1}{2}\left(\frac{k-1}{k-2}\right)}
            \right]}\\
            &= & c_{5.3}\sum_{k=3}^{n}{\frac{1}{\sqrt{2}}\log^\frac{1}{2}\left(\frac{k-1}{k-2}\right)}+
            c_{5.3}\frac{R(n)}{\log^\frac{1}{2}\left(\frac{n}{n-1}\right)},
        \end{array}
        \nonumber
    \end{equation}
    where for the inequality we used the fact that $R(k)$ is positive (due to \eqref{distance_prop_1}). Since $\frac{1}{2{k-2}}\leq  \log\left(\frac{k-1}{k-2}\right)=\log\left(1+\frac{1}{k-2}\right)<\frac{1}{k-2}$ this can be bounded by
    \begin{equation}
        \frac{c_{5.3}}{\sqrt{2}}\sum_{k=3}^{n}{\frac{1}{\sqrt{k-2}}}+\sqrt{2dc_3}R(n)\sqrt{n-1} \leq c_{5.4}\cdot(1+R(n))\sqrt{n-2},
    \nonumber
    \end{equation}
    with $c_{5.4}=c_{5.4}(\om)$. Combining all of the above we get that
    \begin{equation}
        c_{5.1}(d)\cdot e^{R(n)}\sqrt{n-1}\leq M(n)\leq c_{5.4}(\om) (1+R(n))\sqrt{n-2},
        \nonumber
    \end{equation}
    which implies that $R(n)$ is a bounded function $P$ almost surely. Thus one can find two random variables $c_{5.5},c_{5.6}:\Om_0\rightarrow\BR$, which are $P$ almost surely finite and positive, such that
    \begin{equation}
        c_{5.5}\sqrt{n} \leq M(n) \leq c_{5.6}\sqrt{n}. \nonumber
    \end{equation}
    Recalling the definition of $M(n)$, this yields the result.
\end{proof}


\section{Corrector - Construction and harmonicity}\label{sec:corrector}

In this section, we adapt the construction of the corrector presented in \cite{BB06} to our model. The corrector, originated in a paper by Kipnis and Varadhan (see \cite{KV86}) gives a decomposition of random variables into a martingale and a part which is $o(\sqrt{n})$. In our case, as in \cite{Bb06}, this is used to construct a graph deformation (perturbation of the graph embedding in $\mathbb{R}^d$) such that the resulting graph is harmonic, i.e., the location of each vertex is the averaged location of its neighbors and such that the change in location of each point $x\in\mathbb{Z}^d$ is $o(\|x\|_2)$.

\ifarxiv
\else
Since the proofs are very similar to the ones in \cite{BB06} we only state most of the theorems. A more detailed version of this section (including proofs) can be found in the Arxiv version \cite{RR10}.
\fi

We start with the following observation concerning the Markov chain "on environments".

\begin{lem}
    For every bounded measurable function $f:\Om_0\rightarrow \BR$ and every $x\in\BZ^2$ we have
    \begin{equation}
        \BE_P\left[(f\circ \gth_x)\ind_{\{x\in N_0(\om)\}}\right]=\BE_P[f\ind_{\{-x\in N_0(\om)\}}].
        \label{lem_chain_environments}
    \end{equation}
    As a consequence, $P$ is reversible and, in particular, stationary w.r.t the Markov kernel $\Lambda$ defined in \eqref{Lambda}.
\end{lem}

\begin{proof}
    Multiplying \eqref{lem_chain_environments} by $\BP(\Om_0)$ gives
    \begin{equation}
        \BE_Q[f\circ\gth_x\ind_{\Om_0}\ind_{\{x\in N_0(\om)\}}]=\BE_Q[f\ind_{\Om_0}\ind_{\{-x\in N_0(\om)\}}].
        \label{lem_chain_environments2}
    \end{equation}
    The last equality holds since $\ind_{\{x\in N_0(\om)\}}\ind_{\Om_0}=\left(\ind_{\{-x\in N_0(\om)\}}\ind_{\Om_0}\right)\circ\gth_x$ and therefore $f\circ\gth_x\ind_{\Om_0}\ind_{\{x\in N_0(\om)\}}=\left(f\ind_{\Om_0}\ind_{\{-x\in N_0(\om)\}}\right)\circ\gth_x$. Thus taking expectation w.r.t $Q$ and recalling it is shift invariant gives \eqref{lem_chain_environments2}.

    For a measurable function $f:\Om\rightarrow\BR$ define $\Lambda f:\Om_0\rightarrow\BR$ by
    \begin{equation}
        (\Lambda f)(\om)=\frac{1}{2d}\sum_{x\in\BZ^d}{\left(\ind_{\{x\in N_0(\om)\}}f(\gth_x\om)\right)}. \label{Lambda_function}
    \end{equation}
    Using \eqref{lem_chain_environments} we deduce that for any bounded measurable functions $f,g:\Om\rightarrow\BR$,
    \begin{equation}
        \begin{array}{rcl}
            \BE_P[f\cdot (\Lambda g)]&=&\frac{1}{2d}\sum_{x\in\BZ^d}\BE_P[f\cdot (g\circ\gth_x)\ind_{\{x\in N_0(\om)\}}]\\
            &=&\frac{1}{2d}\sum_{x\in\BZ^d}\BE_P[f\circ \gth_{-x}\ind_{\{-x\in N_0(\om)\}}\cdot g]\\
            &=&\frac{1}{2d}\sum_{-x\in\BZ^d}\BE_P[f\circ \gth_{x}\ind_{\{x\in N_0(\om)\}}\cdot g]=\BE_P[(\Lambda f)\cdot g],
        \end{array}
        \label{inner_product_property}
    \end{equation}
    which is the definition of reversibility. Taking $f=1$ and noticing that $\Lambda f=1$, we get that $\BE_P[\Lambda g]=\BE_P[g]$ for every bounded measurable function $g:\Om\rightarrow\BR$, i.e., $P$ is stationary with respect to the Markov kernel $\Lambda$.
\end{proof}


\subsection{The Kipnis-Varadhan Construction}
$~$\\

We can now adapt the construction of the corrector to the present situation. Let $L^2=L^2(\Om_0,\FB,P)$ be the space of all Borel-measurable square integrable functions on $\Om_0$. We use the notation $L^2$ both for $\BR$-valued functions as well as for $\BR^d$-valued functions. We equip $L^2$ with the inner product $\langle f,g\rangle=\BE_P[fg]$, when for vector valued functions on $\Om$ we interpret $"fg"$ as the scalar product of $f$ and $g$. Let $\Lambda$ be the operator  defined by (\ref{Lambda_function}), and expand the definition to vector valued functions by letting $\Lambda$ act like a scalar, i.e., independently on each  component. From (\ref{inner_product_property}) we get that
\begin{equation}
    \langle f,\Lambda g\rangle=\langle\Lambda f,g\rangle,
\end{equation}
and so $\Lambda$ is self adjoint. In addition, for every $f\in L^2$ we have
\[
    |\langle f,\Lambda f\rangle|\leq\frac{1}{2d}\sum_{x\in\BZ^d}{|\langle f,\ind_{\{x\in N_0(\om)\}}f\circ\gth_x\rangle|}=\frac{1}{2d}\sum_{x\in\BZ^d}{|\langle f\ind_{\{x\in N_0(\om)\}},\ind_{\{x\in N_0(\om)\}}f\circ\gth_x\rangle|}
\]
which by the Cauchy-Schwarz inequality can be bounded by
\begin{equation}
    \begin{array}{cl}
        &\frac{1}{2d}\sum_{x\in\BZ^d}{\langle f\ind_{\{x\in N_0(\om)\}},f\ind_{\{x\in N_0(\om)\}}\rangle^{1/2}\cdot\langle\ind_{\{x\in N_0(\om)\}}f\circ\gth_x,\ind_{\{x\in N_0(\om)\}}f\circ\gth_x\rangle^{1/2}}\\
        &~\\
        =&\frac{1}{2d}\sum_{x\in\BZ^d}{\langle f,f\ind_{\{x\in N_0(\om)\}}\rangle^{1/2}\cdot\langle 1,\ind_{\{x\in N_0(\om)\}}f^2\circ\gth_x\rangle^{1/2}},
    \end{array}
    \nonumber
\end{equation}
and by \eqref{lem_chain_environments} equals
\[
    \frac{1}{2d}\sum_{x\in\BZ^d}{\langle f,f\ind_{\{x\in N_0(\om)\}}\rangle^{1/2}\cdot\langle f,f\ind_{\{-x\in N_0(\om)\}}\rangle^{1/2}}\leq\frac{1}{2d}\sum_{x\in\BZ^d}{\langle f,f\ind_{\{x\in N_0(\om)\}}\rangle}=\langle f,f\rangle.
\]
Thus $~\|\Lambda\|_{L^2}\leq 1$. In particular, $\Lambda$ is self adjoint and $sp(\Lambda)\subseteq [-1,1]$.

Let $V:\Om_0\rightarrow\BR^d$ be the local drift at the origin, i.e.,
\begin{equation}
    V(\om)=\frac{1}{2d}\sum_{x\in\BZ^d}{x\ind_{\{x\in N_0(\om)\}}}. \label{drift_equation}
\end{equation}
If the second moment of $f_e$ exists for every $e\in\CE$, then
\[
    \langle V,V\rangle=\sum_{e\in\CE}{\langle V\cdot e,V\cdot e\rangle}=\frac{1}{2d}\BE_P[(V\cdot e)^2]=\frac{1}{2d}\BE_P[f_e^2+f_{-e}^2]<\infty,
\]
and therefore $V\in L^2$. Thus for each $\ep>0$ we can define $\psi_\ep:\Om_0\rightarrow\BR^d$ as the solution in $L^2$ of
\begin{equation}
    (1+\ep-\Lambda)\psi_\ep=V. \label{defining_psi_ep}
\end{equation}

\begin{rem}
This is well defined since the spectrum of $\Lambda$, denoted by $sp(\Lambda)$, is contained in the interval $[-1,1]$, and therefore $sp(1+\ep+\Lambda)\subset[\ep,2+\ep]$. In particular since $\ep>0$ the operator $1+\ep-\Lambda$ has a bounded inverse.
\end{rem}

The following Theorem is the main result concerning the corrector:

\begin{thm}
    There is a function $\chi:\BZ^d\times\Om_0\rightarrow\BR^d$ such that for every $x\in\BZ^d$,
    \begin{equation}
        \lim_{\ep\downarrow 0}{\ind_{\{x\in\CP(\om)\}}(\psi_\ep\circ\gth_x-\psi_\ep)}=\chi(x,\cdot),\quad \text{in}~L^2.
    \end{equation}
    Moreover, the following properties hold:
    \begin{itemize}

        \item (Shift invariance) For $P$ almost every $\om\in\Om_0$
        \begin{equation}
            \chi(x,\om)-\chi(y,\om)=\chi(x-y,\gth_y(\om)), \label{shift_invariance}
        \end{equation}
        for all $x,y\in\CP(\om)$.

        \item (Harmonicity) For $P$ almost every $\om\in\Om_0$, the function
        \begin{equation}
            x\mapsto\chi(x,\om)+x,
        \end{equation}
        is harmonic with respect to the transition probability given in (\ref{transition_probability})

        \item (Square integrability) There exists a constant $C<\infty$ such that
        \begin{equation}
            \|[\chi(x+y,\cdot)-\chi(x,\cdot)]\ind_{\{x\in\CP(\om)\}}(\ind_{\{y\in N_0(\om)\}}\circ\gth_x)\|_2<C, \label{square_integrability}
        \end{equation}
        for all $x,y\in\BZ^d$.
    \end{itemize}
    \label{corrector_thm}
\end{thm}

\ifarxiv
\else

The proof of Theorem \ref{corrector_thm} follows the same lines as the one in \cite{BB06} without any major changes, and therefore we omit it. The following Lemma summarizes few of the intermediate steps in the proof of Theorem \ref{corrector_thm} which will be needed in order to prove the high dimensional CLT.

\begin{lem}
    Let $\psi_\ep$ be defined as in \eqref{defining_psi_ep}, i.e., the solution of $(1+\ep-\Lambda)\psi_\ep=V$. Then
    \begin{equation}
        \lim_{\ep\downarrow 0}{\ep\|\psi_\ep\|_2^2}=0. \label{psi_squared_limit}
    \end{equation}
    For every $x\in\BZ^d$ define
    \begin{equation}
        G_x^{(\ep)}(\om)=\ind_{\Om_0}(\om)\cdot\ind_{\{x\in N_0(\om)\}}(\om)\cdot (\psi_\ep\circ\gth_x(\om)-\psi_\ep(\om)).\label{G_definition}
    \end{equation}
    Then
    \begin{equation}
        \lim_{\ep_1,\ep_2\downarrow 0}{\|G_x^{(\ep_1)}\circ\gth_y-G_x^{(\ep_2)}\circ\gth_y\|_2}=0,\quad \forall x,y\in \BZ^d.\label{lim_of_G_exist}
    \end{equation}
    The corrector is now defined by
    \begin{equation}
        \chi(x,\om)\overset{def}{=}\sum_{k=0}^{n-1}{G_{x_k,x_{k+1}}(\om)},\label{chi_definition}
    \end{equation}
    where $(x_0,x_1,\ldots,x_n)$ is any "coordinate nearest neighbor" path in $\CP(\om)$ from $0$ to $x$ and $G_{x,y}(\om)=\lim_{\ep\downarrow 0}G_x^{(\ep)}\circ \theta_y(\om)$ in the $L^2$ sense.
\end{lem}

\begin{rem}
    The fact that all the limits in the above lemma exist and that the corrector is well defined are all part of the proof of Theorem \ref{corrector_thm}.
\end{rem}

\fi



\ifarxiv
\else
\section{Essential sublinearity of the corrector }
\label{sec:essential_sublinearity}
\fi

Fix $e\in\CE$ and define the random sequence $n^e_k(\om)$ inductively by $n^e_1(\om)=f_e(\om)$ and $n^e_{k+1}=n^e_k(\si_e(\om))$, where $\si_e$ is the induced
translation defined by $\si_e=\gth^{f_e(\om)}_e$. The numbers $n^e_k$ are well-defined and finite $P$ almost surely. Let $\chi$ be the corrector from Theorem \ref{corrector_thm}. The \ifarxiv\else first \fi goal of this section is to prove the following theorem:

\begin{thm}
    For $P$ almost all $\om\in\Om_0$
    \begin{equation}
        \lim_{k\rightarrow\infty}{\frac{\chi(n^e_k(\om)e,\om)}{k}}=0.
    \end{equation}
    \label{coordinate_sublinearity}
\end{thm}

The proof of this theorem is based on the following properties of $\chi(n^e_k(\om)e,\om)$:

\begin{prop}
    ~\\\vspace{-4mm}
    \begin{enumerate}

        \item $\BE_P\big[|\chi(n^e_1(\om)e,\cdot)|\big]<\infty.$

        \item $\BE_P\big[\chi(n^e_1(\om)e,\cdot)\big]=0.$

    \end{enumerate}
    \label{chi_moment_prop}
\end{prop}

\begin{proof}
    Using the definition of the corrector \eqref{chi_definition}, it follows that
    \begin{equation}
        \chi(n^e_1(\om)e,\om)=G_{0,n^e_1(\om)e}(\om).
    \end{equation}
    By \eqref{lim_of_G_exist}, and since $G_{0,n^e_1(\om)e}(\om)$ is the $\ep\downarrow 0$ limit of $G^{(\ep)}_{n^e_1(\om)e}$ in $L^2$, it follows that $G_{0,n^e_1(\om)e}(\om)\in L^2$. Since $P$ is a probability measure, it is in particular a finite measure, and therefore for every $1\leq r<2$ it is also true that $G_{0,n^e_1(\om)e}(\om)\in L^r$. Taking $r=1$ gives
    \begin{equation}
        \BE_P\big[|\chi(n^e_1(\om)e,\cdot)|\big]=\BE_P\big[|G_{0,n^e_1(\om)e}(\om)|\big]<\infty.
    \end{equation}

    For (2), observe that by Definition \ref{G_definition} and Theorem \ref{induced_shift_ergodicy}, for every $\ep>0$,
    \begin{equation}
    \begin{aligned}
        \BE_P\big[G^{(\ep)}_{n^e_1(\om)e}\big]&=\BE_P\big[\ind_{\Om_0}\ind_{\{n^e_1(\om)e\in
        N_0(\om)\}}(\psi_\ep\circ\gth^{n^e_1(\om)}_e-\psi_\ep)\big]\nonumber\\
        &=\BE_P\big[\ind_{\Om_0}\ind_{\{n^e_1(\om)e\in
        N_0(\om)\}}\psi_\ep\circ\gth^{n^e_1(\om)}_e\big]-\BE_P\big[\ind_{\Om_0}\ind_{\{n^e_1(\om)e\in N_0(\om)\}}\psi_\ep\big]\nonumber\\
        &=\BE_P\big[(\ind_{\Om_0}\ind_{\{n^e_1(\om)e\in N_0(\om)\}}\psi_\ep)\circ\si_e\big]-\BE_P\big[\ind_{\Om_0}\ind_{\{n^e_1(\om)e\in
        N_0(\om)\}}\psi_\ep\big]=0.
    \end{aligned}
    \end{equation}
    Thus by the definition of $\chi$ and the fact that it is in $L^1$
    \[
        \BE_P\big[\chi(n^e_1(\om)e,\cdot)\big]=\BE_P\big[G_{0,n^e_1(\om)e}\big]=\lim_{\ep\downarrow 0}\BE_P\big[G^{(\ep)}_{n^e_1(\om)e}\big]=0.
    \]
\end{proof}

\begin{proof}[Proof of Theorem \ref{coordinate_sublinearity}]
    Define $g:\Om\rightarrow\BR^d$ by $g(\om)=\chi(n^e_1(\om)e,\om)$, and let $\si_e$ be the induced shift in direction e. Then
    \begin{equation}
        \chi(n^e_k(\om)e,\om)=\sum_{i=0}^{k-1}{g\circ\si_e^i(\om)}.
    \end{equation}
    By Proposition \ref{chi_moment_prop} we have that $g\in L^1$ and $\BE_P[g]=0$. Since Theorem \ref{induced_shift_ergodicy} ensures $\si_e$ is $P$ preserving and ergodic, the claim follows from Birkhoff's Ergodic Theorem.
\end{proof}

\ifarxiv
\else

Next we turn to discuss general sublinearity of the corrector. The following Theorem states a weaker notion of sublinearity satisfied by the corrector. This notion though weaker than the one obtained for points along coordinate direction is enough in order to prove high dimensional CLT.

\begin{thm}
    For every $\ep>0$ and $P$ almost every $\om\in\Om_0$
    \begin{equation}
        \limsup_{n\rightarrow\infty}{\frac{1}{(2n+1)^d}\sum_{x\in\CP(\om),~|x|\leq n}\ind_{\{|\chi(x,\om)|\geq \ep n\}}}\leq \ep. \label{sublinearity_formula}
    \end{equation}
    \label{sub_linearity_everywhere}
\end{thm}

The proof of Theorem \ref{sub_linearity_everywhere} follows the same lines as the one in \cite{BB06} (Theorem 5.4) without major changes, and therefore we omit it from this version.
\fi



\section{High dimensional Central Limit Theorem}\label{sec:high_dim_CLT}

Here we finally prove the high dimensional CLT, starting with the following Lemma:

\begin{lem}
    Fix $\om\in\Om_0$ and let $x\mapsto \chi(x,\om)$ be the corrector as defined in Theorem \ref{corrector_thm}. Given a path of a random walk $\{X_n\}_{n=0}^{\infty}$ on $\CP(\om)$ with transition probabilities \eqref{transition_probability} let
    \begin{equation}
        M_n^{(\om)}=X_n+\chi(X_n,\om),\quad\forall n\geq 0. \label{deformed_random_walk}
    \end{equation}
    Then $\{M_n^{(\om)}\}_{n\geq 0}$ is an $L^2$-martingale w.r.t the filtration $\{\si(X_0,X_1,\ldots,X_n)\}_{n\geq 0}$. Moreover, conditioned on $X_{k_0}=x$, the increments $\{M_{k+k_0}^{(\om)}-M_{k_0}^{(\om)}\}_{k\geq 0}$ have the same law as $\{M_k^{(\gth_x\om)}\}_{k\geq 0}$.
\end{lem}

\begin{proof}
    Since $X_n$ is bounded, $\chi(X_n,\om)$ is bounded and so $M_n^{(\om)}$ is square integrable with respect to $P_\om$. By Theorem \ref{corrector_thm} the map $x\mapsto x+\chi(x,\om)$ is harmonic with respect to the transition probabilities in \eqref{transition_probability}, and therefore
    \begin{equation}
        E_\om[M_{n+1}^{(\om)}|\si(X_n)]=M_n^{(\om)},\quad\forall n\geq 0,~ P_\om~a.s.
    \end{equation}
    By the definition of $M_n^{(\om)}$ it is $\si(\{X_k\}_{k=1}^n)$-measurable, and therefore $\{M_n^{(\om)}\}$ is a martingale. The stated relation between the laws of $\{M_{k+k_0}^{(\om)}-M_{k_0}^{(\om)}\}_{k\geq 0}$ and $\{M_k^{(\gth_x \om)}\}_{k\geq 0}$ is implied by the shift invariance proved in Theorem  \ref{corrector_thm} and the fact that $\{M_n^{(\om)}\}_{n\geq 0}$ is a simple random walk on the deformed graph.
\end{proof}

\begin{thm}[CLT of the Modified random walk]
    Fix $d\geq 2$. and assume $P$ satisfies assumptions \ref{Assumptions} and \ref{assumption3}. For $\om\in\Om_0$ let $\{X_n\}_{n\geq 0}$ be a random walk with transition probabilities \eqref{transition_probability} and $\{M_n^{(\om)}\}_{n\geq 0}$ as in \eqref{deformed_random_walk}. Then for $P$ almost every  $\om\in\Om_0$ we have
    \begin{equation}
        \lim_{n\rightarrow\infty}\frac{M_n^{(\om)}}{\sqrt{n}}\overset{D}{=}N(0,D),
    \end{equation}
    where the convergence is in distribution and $N(0,D)$ is a $d$-dimensional multivariate normal distribution with covariance matrix $D$ which depends on $d$ and the distribution $P$, given by $D_{i,j}=\BE\left[cov(M_1^{(\om)}\cdot e_i,M_1^{(\om)}\cdot e_j)\right]$. \label{Modified_CLT}
\end{thm}

\begin{proof}
    Let
    \begin{equation}
        V_n^{(\om)}(\ep)=\frac{1}{n}\sum_{k=0}^{n-1}{E_\om\left[D_k^{(\om)} \ind_{\{\min_{i,j}|(D_k^{(\om)})_{i,j}|\geq \ep \sqrt{n} \}}\Big|X_0,X_1,\ldots,X_k\right]},\nonumber
    \end{equation}
    where $D_k^{(\om)}$ is the covariance matrix for $M_{k+1}^{(\om)}-M_k^{(\om)}$. By the Lindeberg-Feller Central Limit Theorem (see for example \cite{Du96}, Theorem 4.5),  it is enough to show that\\
    \begin{enumerate}

        \item $\lim_{n\rightarrow\infty}{V_n^{(\om)}(0)}=D$ in $P_\om$ probability. \\

        \item $\lim_{n\rightarrow\infty}{V_n^{(\om)}(\ep)}=0$ in $P_\om$ probability for every $\ep>0$.\\

    \end{enumerate}
    Both conditions are implied from Theorem \ref{Thm_mutual_ergodic}. Indeed, one can write $V_n^{(\om)}(0)$ as
    \begin{equation}
        V_n^{(\om)}(0)=\frac{1}{n}\sum_{k=0}^{n-1}{h_0\circ\gth_{X_k}(\om)}, \nonumber
    \end{equation}
    where
    \begin{equation}
        h_K(\om)=E_\om\left[D_1^{(\om)} \ind_{\{\min_{i,j}|(D_1^{(\om)})_{i,j}|\geq K \}}\right].\nonumber
    \end{equation}
    Therefore by Theorem \ref{Thm_mutual_ergodic} we have for $P$ almost every $\om\in\Om_0$
    \begin{equation}
        \lim_{n\rightarrow\infty}{V_n^{(\om)}(0)}=\BE\left[h_0(\om)\right]=D.\nonumber
    \end{equation}
    Turning to the second limit, for every $K\in\BR$ and $\ep>0$ it holds that $\ep\sqrt{n}>K$ for sufficiently large $n$, and therefore $f_{\ep\sqrt{N}}\leq f_K$. Consequently, by the Dominated Convergence Theorem
    \begin{equation}
        \limsup_{n\rightarrow\infty}{V_n^{(\om)}(\ep)}\leq \BE\left[D_1^{(\om)} \ind_{\{\min_{i,j}|(D_1^{(\om)})_{i,j}|\geq K \}}\right]\underset{_{K\rightarrow\infty}}{\longrightarrow}~ 0,\quad P~\text{a.s,}\nonumber
    \end{equation}
    where in order to apply the Dominated Convergence Theorem, we used the fact that $M_1^{(\om)}\in L^2$.
\end{proof}

Finally we turn to prove the high dimensional Central Limit Theorem

\begin{proof}[Proof of Theorem \ref{CLT2}]
    Due to Theorem \ref{Modified_CLT} it is enough to prove that for $P$ almost every $\om\in\Om_0$
    \begin{equation}
        \lim_{n\rightarrow\infty}{\frac{\chi(X_n,\om)}{\sqrt{n}}}{\longrightarrow}0,\quad P_\om \text{-in probability}.
    \end{equation}
    This will follow once we show that for some random variable $C=C(\om)$ which is $P$ almost surely finite and positive
    \begin{equation}
        \limsup_{n\rightarrow\infty}{P_\om\left(|\chi(X_n,\om)|>\ep\sqrt{n}\right)}<C\ep^{1/d},\quad \forall \ep>0,~P ~\text{a.s.} \label{final_claim}
    \end{equation}

    Separating the event in \eqref{final_claim} we can bound its probability by
    \begin{equation}
    \begin{aligned}
        P_\om\left(|\chi(X_n,\om)|>\ep\sqrt{n}\right) \nonumber
        \leq P_\om\left(\|X_n\|>\frac{\sqrt{n}}{\ep^{1/d}}\right)+
        P_\om\left(\chi(X_n,\om)>\ep\sqrt{n}~,~\|X_n\|\leq\frac{\sqrt{n}}{\ep^{1/d}}\right)\nonumber
    \end{aligned}
    \end{equation}
    Thus it is enough to deal with each term on the r.h.s separately. For the first term note that by Theorem \ref{asymptotic_X_n} and the Markov inequality, there exists a random variable  $c=c(\om)$, which is $P$ almost surely finite and positive, so that
    \begin{equation}
        P_\om\left[\|X_n\|>\frac{1}{\ep^{1/d}}\sqrt{n}\right]\leq \ep^{1/d}\frac{\BE_\om[\|X_n\|]}{\sqrt{n}}\leq c\ep^{1/d},\quad P ~\text{a.s.}\label{Markov_X_n}
    \end{equation}
    Moving to deal with the second term, by Proposition \ref{bound_of_transitions} we can write
    \begin{equation}
        \begin{array}{rcl}
            P_\om\left(\chi(X_n,\om)>\ep\sqrt{n}~,~\|X_n\|\leq\frac{\sqrt{n}}{\ep^{1/d}}\right)
            &=&\sum_{x\in\CP(\om)}P_\om^n(0,x)\ind_{\left\{|\chi(x,\om)|>\ep\sqrt{n},~x\in\left[-\frac{\sqrt{n}}{\ep^{1/d}},\frac{\sqrt{n}}{\ep^{1/d}}\right]\right\}}\\
            &~&\\
            &\leq & \frac{K}{n^\frac{d}{2}}\sum\limits_{\tiny{\begin{array}{c}x\in \CP(\om)\\ |x|\leq \frac{\sqrt{n}}{\ep^{1/d}}\end{array}}}{\ind_{\left\{\chi(x,\om)>\ep\sqrt{n}\right\}}}\\
            &~&\\
            &=&K\left(\frac{2}{\ep^{1/d}}+\frac{1}{\sqrt{n}}\right)^d\frac{1}{\left(2\frac{\sqrt{n}}{\ep^{1/d}}+1\right)^d}\sum\limits_{\tiny{\begin{array}{c}x\in \CP(\om)\\ |x|\leq \frac{\sqrt{n}}{\ep^{1/d}}\end{array}}}{\ind_{\left\{\chi(x,\om)>\ep^{1+1/d}\frac{\sqrt{n}}{\ep}    \right\}}},\\
        \end{array}
        \nonumber
    \end{equation}
    which by Theorem \ref{sub_linearity_everywhere} yields that
    \begin{equation}
        \limsup_{n\rightarrow\infty}P_\om\left(\chi(X_n,\om)>\ep\sqrt{n}~,~\|X_n\|\leq\frac{\sqrt{n}}{\ep^{1/d}}\right)\leq 2^dK\ep^{1/d}\nonumber
    \end{equation}
    as required.
\end{proof}


\section{Some Conjectures And Questions}\label{sec:conj_and_ques}

While we have full classification of transience-recurrence of random walks on discrete point processes in dimensions $d=1$ and $d\geq 3$, we only have a partial classification in dimension 2. We therefore give the following two conjectures:

\begin{conj}
    There are transient two dimensional random walks on discrete point processes.
\end{conj}

\begin{conj}
    The condition given in Theorem \ref{tran_recu2}, for recurrence of two-dimensional random walk on discrete point process, i.e., the existence of a constant $C>0$ such that
    \begin{equation}
        \sum_{k=N}^{\infty}{\frac{k\cdot P(f_{e_i}=k)}{\BE(f_{e_i})}}\leq\frac{C}{N},\quad i\in\{1,2\},~N\in\BN
    \end{equation}
    is not necessary.
\end{conj}

In Theorem \ref{CLT2} we gave conditions for the random walk on discrete point processes to satisfy a Central Limit Theorem. However, we didn't give any example for a random walk without a Central Limit Theorem. We therefore give the following conjecture:

\begin{conj}
    There are random walks on discrete point processes in high dimensions that don't satisfy a Central Limit Theorem.
\end{conj}

In the proof of Theorem \ref{CLT2} we used the additional assumption that there exists $\ep_0>0$ such that $E_P[f_e^{2+\ep_0}]<\infty$ for every  $e\in\CE$. The assumption that the second moments are finite, is fundamental in our CLT proof in order to build the corrector, and seems to be necessary for the CLT to hold. On  the other hand, existence of such $\ep_0>0$ though needed in our proof, was used only in order to bound \eqref{distance_prop_3}. We therefore give the following  conjecture:

\begin{conj}
    Theorem \ref{CLT2} is true even with the weaker assumption that only the second moments are finite.
\end{conj}

Even if the Theorem is true with the weaker assumption that only the second moment of the distances between points is finite, we can still ask the following question:

\begin{ques}
    Can one find examples for random walks on discrete point processes that satisfy a Central Limit Theorem in high dimensions but don't have all of their second moments finite?
\end{ques}

We also have the following conjecture about the Central Limit Theorem:

\begin{conj}
	Theorem \ref{CLT2} can be strengthened as follows: Let $(\Omega,\mathcal{B},Q)$ be a 		$d$-dimensional discrete point process satisfying assumptions \ref{Assumptions} and 		\ref{assumption3}. Then for $P$ almost every $\om\in\Om_0$ the random walk satisfies an 	invariance principle (i.e., converges to Brownian motion under appropriate scaling). 
\end{conj}

Our model describes non nearest neighbors random walk on random subset of $\BZ^d$ with uniform transition probabilities. We suggest the following generalization  of the model:

\begin{ques}
    Fix $\alpha\in\BR$. We look on the same model for the environments with transition probabilities as follows: for $\om\in\Om_0$
    \begin{equation}
        P_\om(X_{n+1}=u|X_n=v)=\left\{
        \begin{array}{cc}
            0&~~~u\notin N_v(\om) \\
            \frac{1}{Z(v)}\|u-v\|^\alpha&~~~u\in N_v(\om)
        \end{array}
        \right. ,
    \end{equation}
    where $Z(v)$ is normalization constant (The case $\alpha=0$ is the uniform distribution case). Which of the Theorems proved in this paper can be generalized to the extended model?
\end{ques}


\bigskip
\bigskip

\paragraph{Acknowledgements.}
The authors would like to thank an anonymous referee for the careful reading of this paper and many helpful comments. Research of N. B. and R.R. was partially supported by ERC StG grant 239990.

\bigskip


\bibliography{pointproc}
\bibliographystyle{alpha}


\end{document}